\documentclass[12pt]{iopart}%
\usepackage{iopams}
\usepackage{setstack}
\usepackage{mathrsfs}
\usepackage{cite}

\newtheorem{theorem}{Theorem}[section]

\newtheorem{corollary}[theorem]{Corollary}
\newtheorem{definition}[theorem]{Definition}
\newtheorem{example}[theorem]{Example}
\newtheorem{lemma}[theorem]{Lemma}

\newtheorem{proposition}[theorem]{Proposition}
\newtheorem{remark}[theorem]{Remark}

\newenvironment{proof}{\smallskip\noindent{\it Proof.}\rm}
                        {\hspace*{\fill} $\Box$\medskip}
\newenvironment{proofof}[1]{\smallskip\noindent{\bf #1}\rm}
                {\hspace*{\fill} $\Box$\medskip}

\eqnobysec

\newcommand\myIm{\mathrm{Im\,}}
\newcommand\myRe{\mathrm{Re\,}}
\newcommand\Pos{\mathrm{Pos\,}}
\newcommand\Ran{\mathrm{Ran\,}}
\newcommand\ad{\mathrm{ad}}
\newcommand\rmy{\mathrm{y}}

\begin{document}

\title[Inverse scattering for Miura potentials]%
{Inverse scattering for Schr\"odinger operators with Miura potentials, I.
Unique Riccati representatives and ZS-AKNS systems}

\author{C Frayer$^1$, R O Hryniv$^{2,3,4}$, Ya V Mykytyuk$^4$ and P A Perry$^5$}

\address{$^1$ Department of Mathematics, University of Wisconsin--Platteville,
Platteville, WI 53818, U.\,S.\,A.}

\address{$^2$ Institute for Applied Problems of Mechanics and Mathematics, 3b Naukova str., 79601 Lviv, Ukraine}
\address{$^3$ Institute of Mathematics, the University of Rzesz\'{o}w, 16\,A Rejtana str., 35-959 Rzesz\'{o}w, Poland}
\address{$^4$ Department of Mechanics and Mathematics, Lviv Franko National University, 1~Uni\-ver\-sy\-tets'ka str., 79602, Lviv, Ukraine}
\address{$^5$ Department of Mathematics, University of Kentucky, Lexington, KY 40506-0027, U.\,S.\,A.}

\eads{\mailto{frayerc@uwplatt.edu}, \mailto{rhryniv@iapmm.lviv.ua}, \mailto{yamykytyuk@yahoo.com}, \mailto{perry@ms.uky.edu}}


\begin{abstract}
This is the first in a series of papers on scattering theory for one-dimensional Schr\"{o}dinger operators with highly singular
potentials~$q\in H_{\mathrm{loc}}^{-1}(\mathbb{R})$.
In this paper, we study Miura potentials~$q$ associated to positive Schr\"{o}dinger operators that admit a Riccati representation~$q=u^{\prime}+u^{2}$ for a unique
 $u\in L^{1}(\mathbb{R})\cap L^{2}(\mathbb{R})$.
Such potentials have a well-defined reflection coefficient~$r(k)$ that satisfies~$|r(k)|<1$ and determines~$u$ uniquely. We show that the scattering map~$\mathcal{S}\,:\, u\mapsto r$ is real-analytic with real-analytic inverse. To do so, we exploit a natural complexification of the scattering map associated with the ZS-AKNS\ system. In subsequent papers, we will consider larger classes of potentials including singular potentials with bound states.
\end{abstract}







\section{Introduction}

This is the first of a series of papers on scattering theory for
one-dimensional Schr\"{o}dinger operators
\[
    S:= - \frac{\rmd^2}{\rmd x^2} + q
\]
on the line with highly singular potentials~$q$. Our goal is to extend the inverse scattering method as discussed e.g. in~\cite{Levitan:1987,DT:1979,Marchenko:1986,Faddeev:1959,Melin:1985} in order to study initial value problems for completely integrable dispersive equations with highly singular initial data.

It is natural to begin with the case where the Schr\"{o}dinger operator has no
bound states; the corresponding potentials then admit a Riccati representation given by the Miura map~\cite{Miura:1968}. Recall that the \emph{Miura map} is the nonlinear mapping%
\begin{footnote}{See~Subsection~\ref{ssec.notat} for explicit definitions of all the function spaces appearing in the paper}
\end{footnote}
\begin{equation}\label{eq.Miura}
    \eqalign{
    B\,:\, &L_{\mathrm{loc}}^{2}(\mathbb{R})
            \rightarrow H_{\mathrm{loc}}^{-1}(\mathbb{R}),\\
          &\hspace*{30pt} u \mapsto u'+u^{2}.}
\end{equation}
It is not difficult to see~\cite{Miura:1968} that if $u(x,t)$ is a smooth solution of the~mKdV equation, then
\[
    (Bu)(x,t)
        = \frac{\displaystyle{\partial u}}{\displaystyle{\partial x}} (x,t) + u^2(x,t)
\]
is a smooth solution of the~KdV equation. For this reason, the Miura map has played a fundamental role
in the study of existence and well-posedness questions for
these two equations.

The range of the Miura map may be characterized as follows. For
$q\in H_{\mathrm{loc}}^{-1}(\mathbb{R})$ real-valued and $\varphi
\in C_{0}^{\infty}(\mathbb{R})$, consider the Schr\"{o}dinger form
\begin{equation}\label{eq.h}
    \mathfrak{s}(\varphi)
        := \int\left\vert \varphi^{\prime}(x)\right\vert^{2}\rmd x
            + \langle q,\left\vert \varphi\right\vert ^{2}\rangle,
\end{equation}
where $\langle\,\cdot\,,\,\cdot\,\rangle$ is the pairing between $H_{\mathrm{loc}}^{-1}(\mathbb{R})$ and $H^1_{\mathrm{comp}}(\mathbb{R})$.
In~\cite{KPST:2005} it was shown that if $q$ is any real-valued distribution in~$H_{\mathrm{loc}}^{-1}(\mathbb{R})$ for which the Schr\"{o}dinger form~$\mathfrak{s}$ is nonnegative, then $q$ may be presented as $q=B u$ for a function $u\in L_{\mathrm{loc}}^{2}(\mathbb{R})$ that need not be unique. We will call such a potential~$q$ a \emph{Miura potential}, and we will call any function~$u$ with $q=B u$ a \emph{Riccati representative} for~$q$. It is not difficult to see
that two Riccati representatives for a given~$q$ differ by a continuous
function. Any Riccati representative~$u$ is the logarithmic derivative of a
positive distributional solution to the zero-energy Schr\"{o}dinger equation
for $q$.

In this paper, we will study inverse scattering for the subset of the Miura
potentials which admit a Riccati representative $u\in L^{1}(\mathbb{R})\cap
L^{2}(\mathbb{R})$. In this case, such a Riccati representative is unique, so we
may parameterize the potentials by their Riccati representatives.
In Paper~II of this series~\cite{HMP:2008a}, we consider Miura potentials $q$ which do not have a unique Miura representative but instead possess the following property: There exist Riccati representatives $u_+$ and $u_-$ which belong to $L^2(\mathbb{R})$ and, in addition, $u_-$ is integrable on $(-\infty,0)$, while $u_+$ is integrable on $(0,\infty)$. It will turn out that a potential~$q$ in this class is uniquely determined by the data $\left. u_- \right|_{(-\infty,0)}$, $\left.u_+\right|_{(0,\infty)}$, and the ``jump'' $(u_+-u_-)(0)$. This class includes the Faddeev--Marchenko class of potentials $q \in L^1\bigl(\mathbb{R},(1+|x|)\rmd x\bigr)$ as well as many highly oscillatory and distributional potentials. Finally, in Paper~III~\cite{HMP:2008b}, we will show how to add bound states. The class of potentials covered in the present paper is ``non-generic'' in the sense that, even for regular short-range potentials, the Riccati representative is generically non-unique; on the other hand, the ideas used here lay the groundwork for the analysis of ``generic'' singular potentials in Papers~II\ and~III.

The inverse scattering problem for Miura potentials with $u\in L^{1}(\mathbb{R})\cap
L^{2}(\mathbb{R})$ was worked out in detail from the point of view of Schr\"{o}dinger scattering in the thesis of the first author~\cite{Frayer:2008}. In this
paper, we take a somewhat different point of view and exploit a natural
complexification of the problem that leads to the ZS-AKNS\ system (see
Zakharov and Shabat~\cite{ZS:1971} and Ablowitz, Kaup, Newell, and Segur~\cite{AKNS:1974}) associated to the defocusing nonlinear Schr\"{o}dinger
equation. Elsewhere~\cite{BHP:2009}, we will apply the refined continuity
results obtained here to study solutions to the defocusing NLS\ equation with singular initial data.

Let us first describe the connection between the ZS-AKNS\ system and
scattering with Miura potentials. For a real-valued $u$ in the Schwartz class~$S(\mathbb{R})$, suppose that $\Psi$ is a matrix-valued solution of the ZS-AKNS\ system\footnote{Our equation (\ref{eq.AKNS}) is actually a special
case of the ZS-AKNS\ system since in general one takes%
\[
    \mathrm{Q}(x)
        =\left(
            \begin{array}[c]{cc}%
                0 & u_{1}(x)\\
                u_{2}(x) & 0
            \end{array}
        \right)
\]
for complex-valued $u_{1}$ and $u_{2}$.}%
\begin{equation}\label{eq.AKNS}%
    \Psi^{\/\prime} = \rmi z\sigma\Psi+\mathrm{Q}(x)\Psi,
\end{equation}
where
\begin{equation*}
    \sigma  =\frac{1}{2}\sigma_{3}
            :=\left(
                \begin{array}[c]{cc}%
                    \frac{1}{2} & 0\\
                    0 & -\frac{1}{2}%
                \end{array}\right),
\qquad
    \mathrm{Q}(x)    =\left(
                \begin{array}[c]{cc}%
                    0 & u(x)\\
                    u(x) & 0
                \end{array}\right).
\end{equation*}
An easy calculation shows that if~$\Psi$ solves~(\ref{eq.AKNS}) for
a real-valued $u\in S(\mathbb{R})$, then the row vector%
\[
    \chi:=\left(  1 \ 1\right)  \Psi
\]
solves the Schr\"{o}dinger equation%
\[
    -\chi''+q(x)\chi=\frac{z^{2}}{4}\chi,
\]
where%
\begin{equation}\label{eq.qu}%
    q(x)=u^{\prime}(x)+u^{2}(x)
\end{equation}
is the image of~$u$ under the Miura map~(\ref{eq.Miura}). Hence, in order to
study the direct and inverse scattering maps for the Schr\"{o}dinger equation
with a Miura potential, it suffices to study the analogous maps for the ZS-AKNS\ system.

To state the connection between the scattering maps more precisely, suppose
again that~$u\in S(\mathbb{R})$ is real-valued and $s \in \mathbb{R}$. As is well-known (see for example~\cite{BC:1984}, where more general problems under weaker regularity and decay assumptions are studied), there exist solutions $\Psi_{\pm}$ of the ZS-AKNS\ system with respective asymptotics%
\begin{equation}\label{eq.Psipm}%
    \lim_{x\rightarrow\pm\infty}
        | \Psi_{\pm}(x,s)-\exp(\rmi sx\sigma)| =0
\end{equation}
and a unique matrix $\mathrm{A}(s)$ of the form
\begin{equation*}
    \mathrm{A}(s) = \left(
        \begin{array}[c]{cc}%
            a(s) & \overline{b}(s)\\
            b(s) & \overline{a}(s)
        \end{array}\right)
\end{equation*}
with
\[
    \left\vert a(s)\right\vert ^{2}-\left\vert b(s)\right\vert ^{2}=1
\]
such that
\begin{equation}\label{eq.Psipm.A}%
    \Psi_{+}(x,s)=\Psi_{-}(x,s)\mathrm{A}(s).
\end{equation}
If $u\in S(\mathbb{R})$, we have $a-1\in S(\mathbb{R})$ and
$b\in S(\mathbb{R})$. The function $a(s)$ extends to a bounded
analytic function $a(z)$ on $\myIm z >0$ with $|a(z)| \to 1$ as $|z|\to\infty$.

The \emph{left reflection coefficient }\ is given by%
\begin{equation}
r_{-}(s):=b(s)/a(s),\label{eq.r-}%
\end{equation}
and $\left\vert r_{-}(s)\right\vert <1$ strictly if $u\in S(\mathbb{R})$. Similarly, the \emph{right reflection coefficient} is given by
\begin{equation}
r_{+}(s):=-\overline{b}(s)/a(s).\label{eq.r+}%
\end{equation}
These reflection coefficients are identical (up to the change $s\mapsto2s$ in
the spectral parameter) to the reflection coefficients in the Schr\"{o}dinger
problem for the Miura potential (\ref{eq.qu}). Indeed, if \
\[
    \chi_{\pm}:=(1\ 1)\Psi_{\pm}%
\]
and $\chi_{\pm}=\left(  \chi_{1}^{\pm},\chi_{2}^{\pm}\right)$, we can use~(\ref{eq.Psipm}) and~(\ref{eq.Psipm.A}) to conclude that%
\begin{equation}\label{eq.jost+}%
    \chi_{1}^{+}(x,s)\sim
        \left\{
            \begin{array}[c]{cc}%
        a(s)\rme^{\rmi sx/2}+b(s)\rme^{-\rmi sx/2}, \quad
            & x\rightarrow-\infty,\\
        \rme^{\rmi sx/2},
            & x\rightarrow+\infty,
            \end{array}
        \right.
\end{equation}
and
\begin{equation}\label{eq.jost-}%
    \chi_{2}^{-}(x,s)\sim
        \left\{
            \begin{array}[c]{cc}%
        \rme^{-\rmi sx/2},
            & x\rightarrow-\infty,\\
        a(s)\rme^{-\rmi sx/2}-\overline{b}(s)\rme^{\rmi sx/2}, \quad
            & x\rightarrow+\infty.
            \end{array}
        \right.
\end{equation}
Thus, up to the reparametrization, $r_{-}(s)$ is the usual Schr\"{o}dinger
reflection coefficient for scattering from the left, and $r_{+}(s)$ is the
usual Schr\"{o}dinger reflection coefficient for scattering from the right.
The asymptotic relations~(\ref{eq.jost+}) and (\ref{eq.jost-}) show that, up to reparametrization, $\chi_1^+(\cdot,s)$ and $\chi_2^-(\cdot,s)$ are the usual Jost solutions at $+\infty$ and $-\infty$ respectively.

These calculations motivate us to consider the inverse problem for the ZS-AKNS
system, now for $u\in L^{1}(\mathbb{R})\cap L^{2}(\mathbb{R})$. This class of~$u$ is chosen so that the corresponding Miura potentials lie in~$H^{-1}(\mathbb{R})$ and define self-adjoint Schr\"odinger operators~$S$~\cite{HM:2001,KPST:2005}; moreover, they have sufficient decay to allow a well-defined scattering theory. It will
pose no essential difficulty in considering the inverse scattering problem for
the system (\ref{eq.AKNS}) with%
\begin{equation}\label{eq.Q}
    \mathrm{Q}(x)
        = \left(\begin{array}[c]{cc}
            0 & u(x)\\ \overline{u}(x) & 0
          \end{array} \right)
\end{equation}
and $u$ complex-valued; the scattering map for the Schr\"{o}dinger problem
with Miura potential $q=u^{\prime}+u^{2}$ and $u$ real-valued will be the
restriction of the resulting scattering maps to real-valued $u$.

For this more general scattering problem, there still exist matrix-valued Jost solutions~$\Psi_\pm$ and a matrix~$\mathrm{A}$ with the above properties. In particular, the reflection coefficients~$r_\pm$ are well defined, and we introduce the direct scattering maps~$\mathcal{S}_\pm$ via
\begin{eqnarray*}
    \mathcal{S}_{+}  &  :u\rightarrow r_{+},\\ 
    \mathcal{S}_{-}  &  :u\rightarrow r_{-}. 
\end{eqnarray*}
Define the Fourier transforms~$\mathcal{F}_+$ and $\mathcal{F}_-$ via
\begin{eqnarray*}
    (\mathcal{F}_{+}f)(x)  &  :=\frac{1}{2\pi}
        \int_{-\infty}^{\infty} \rme^{\rmi sx} f(s)\,\rmd s, \\ 
    (\mathcal{F}_{-}f)(x)  &  :=\frac{1}{2\pi}
        \int_{-\infty}^{\infty} \rme^{-\rmi sx}f(s)\,\rmd s; 
\end{eqnarray*}
it will be convenient to work with
\begin{eqnarray}
    F_{+}(x)  &:=(\mathcal{F}_{+}r_+)(x) =
             \frac{1}{2\pi}\int_{-\infty}^{\infty}
                    \rme^{\rmi sx}r_{+}(s)\,\rmd s,   \label{eq.F+}\\
    F_{-}(x)  &:=(\mathcal{F}_{-}r_-)(x) =
             \frac{1}{2\pi}\int_{-\infty}^{\infty}
                    \rme^{-\rmi sx}r_{-}(s)\,\rmd s  \label{eq.F-}%
\end{eqnarray}
and to analyze the mappings $u\mapsto F_+$ and $u\mapsto F_-$  given by $\mathcal{F}_{+}\circ\mathcal{S}_{+}$ and $\mathcal{F}_{-}\circ\mathcal{S}_{-}$.

To state our results, let us define the Banach space
\begin{equation}\label{eq.X}%
    X:=L^{1}\mathbb{(R})\cap L^{2}(\mathbb{R})
\end{equation}
with the norm%
\begin{equation}\label{eq.X.norm}%
    \| F \|_{X} := \| F \|_{L^{1}} + \| F \|_{L^{2}}.
\end{equation}
Let
\begin{equation*}
    X_{1} :=\left\{  F\in X \,:\, \|\widehat{F}\|_{\infty}<1\right\}
\end{equation*}
with norm induced from~$X$, where%
\begin{equation}\label{eq.Fourier}
    \widehat{F}(s):= \int_{-\infty}^{\infty} \rme^{-\rmi sx}F(x)\,\rmd x
        = \bigl(\mathcal{F}_+^{-1} F\bigr)(s).
\end{equation}
Our first result is:

\begin{theorem}
\label{thm.1}The mappings $\mathcal{F}_{+}\circ\mathcal{S}_{+}$ and
$\mathcal{F}_{-}\circ\mathcal{S}_{-}$ are continuous bijections between $X$
and $X_{1}$. The inverse mappings take the form
\[
\left(  \mathcal{F}_{\pm}\circ\mathcal{S}_{\pm}\right)  ^{-1}=I+\mathit{\Phi}_{\pm}%
\]
where $I$ is the identity map and \
\[
\mathit{\Phi}_{\pm}:X_1\rightarrow L^{1}(\mathbb{R})\cap C(\mathbb{R})
\]
are continuous maps. Moreover, both the direct maps $\mathcal{F}_{\pm}\circ\mathcal{S}_{\pm}$ and their inverses are analytic in the sense of Definition~\ref{def.anal} below.
\end{theorem}

We note that the mappings of the above theorem are not analytic in the usual sense as, for example, the direct maps depend on $u$ and $\overline{u}$ and the inverse maps depend on $F$ and $\overline{F}$.

Now let $X_{\mathbb{R}}$ be the restriction of $X$ to real-valued functions,
and let $\left(  X_{\mathbb{R}}\right)  _{1}$ be the analogous restriction of
$X_{1}$. If $u$ is real-valued, the reflection coefficients obey the
additional constraint%
\begin{equation}\label{eq.r.real}%
    \overline{r_{\pm}(s)}=r_{\pm}(-s),
\end{equation}
which implies that $F_{+}$ and $F_{-}$ are real-valued. Thus, the maps
$\mathcal{F}_{\pm}\circ\mathcal{S}_{\pm}\,\ $\ restrict to maps from $X_{\mathbb{R}%
}$ to $\left(  X_{\mathbb{R}}\right)  _{1}$. From Theorem \ref{thm.1} and the
analyticity properties of the direct and inverse maps, we obtain our main result on inverse scattering for Schr\"odinger operators with Miura potentials:

\begin{corollary}\label{cor.Schroed}
The maps $\mathcal{F}_{+}\circ\mathcal{S}_{+}$ and $\mathcal{F}_{-}%
\circ\mathcal{S}_{-}$ restrict to real-analytic bijections between
$X_{\mathbb{R}}$ and~$\left(  X_{\mathbb{R}}\right)_{1}$. Moreover,
\[
    \left( \mathcal{F}_{\pm}\circ\mathcal{S}_{\pm}\right)^{-1}=I+\mathit{\Phi}_{\pm}%
\]
where $I$ is the identity map and
\[
    \mathit{\Phi}_{\pm}\,:\, \left(  X_{\mathbb{R}}\right)_{1} \rightarrow
                    L^{1}(\mathbb{R})\cap C(\mathbb{R}).
\]
\end{corollary}

\begin{remark}
The class of Miura potentials with a Riccati representative $u\in X_{\mathbb R}$ is ``non-generic'' among  potentials generating non-negative Schr\"{o}dinger operators: even for regular potentials, the generic case is that $r(0)=-1$ (see, e.g., \cite{DT:1979}, Theorem~1, p.~146 and Remark~9, pp.~152--153, and Example~\ref{ex.delta} below). In the generic case, the scattering problem can be parametrized by left and right Riccati representatives in $L^{2}(\mathbb R)$ having additional
integrability on left and right half-lines~\cite{HMP:2008a}.
\end{remark}

In the proof of Theorem \ref{thm.1}, we reconstruct $u$ on a half-line
$(c,\infty)$ using $r_{+}$ and a ``right''
Gelfand--Levitan--Marchenko equation, and we reconstruct $u$ on a half-line
$(-\infty,c)$ using a ``left''
Gelfand--Levitan--Marchenko equation, and show that these reconstructions give
continuous maps from $X_{1}$ into the respective spaces
 $X_{c}^{+}:=L^{1}(c,\infty) \cap L^{2}(c,\infty) $ and
 $X_{c}^{-}:=L^{1}(-\infty,c)\cap L^{2}(-\infty,c)$.
We then exploit the existence of an involution~$\mathcal{I}$ that intertwines $\mathcal{S}_{+}$ and $\mathcal{S}_{-}$ to obtain continuity of the inverse maps into the full space~$X$.

Inverse scattering for the AKNS\ system has been extensively studied: see for
example the original papers~\cite{ZS:1971} and~\cite{AKNS:1974}, Shabat's
solution of the inverse problem for potentials in~$L^{1}(\mathbb R)$~\cite{Shabat:1975,Shabat:1979}, the works of
Beals and Coifman~\cite{BC:1984,BC:1985,BC:1987}, Cohen and Kappeler~\cite{CK:1992}, Wadati~\cite{Wadati:1973}, and the monographs of
Faddeev and Takhtajan~\cite{FT:2007}, and Beals, Deift, and Tomei~\cite{BDT:1988}.
Aktosun, Klaus, and van der Mee~\cite{AKM:2000} investigated direct and inverse scattering for AKNS systems on the line with integrable matrix-valued potentials and, in particular, derived partial characterization of the scattering data. Earlier in~\cite{AKM:1993}, they extended the classical scattering theory for Schr\"odinger operators by treating a class of singular potentials constructed via the Darboux transformations that lead to ambiguities in the inverse scattering; see also the related paper by Degasperis and Sabatier~\cite{DS:1987}.
Recently, Demontis and van der Mee~\cite{Dvdm:2008} studied scattering theory
and wave operators for AKNS-type equations including those considered here and
completely characterized scattering data for potentials $\mathrm{Q}\in L^{1}(\mathbb{R})$. Novikov \cite{Novikov:1997} studied inverse scattering for
$u$ in $L^{1}$-based Sobolev spaces $W^{n,1}(\mathbb R)$ with $n\geq1$, and showed that the leading singularities of $u$ (modulo $W^{n-1,1}(\mathbb R)$) can be recovered from the scattering coefficients via the Fourier transform. In a similar way, Theorem~\ref{thm.1} can be interpreted as saying that the singularities of $u$
are recovered modulo $C(\mathbb{R})$ from the Fourier transform of the
reflection coefficient. Xin Zhou \cite{Zhou:1998} used Riemann--Hilbert
techniques to show that the direct scattering maps $\mathcal{S}_\pm:u\mapsto r_\pm$ are
bijections from $H^{j,k}(\mathbb{R)}$ to $H_{1}^{k,j}(\mathbb{R)}$ where
$H^{j,k}(\mathbb{R)}$ is the $L^{2}$ weighted Sobolev space of functions with
$j$ derivatives belonging
    to~$L^{2}\bigl(\mathbb{R},\bigl(1+|x|^{2}\bigr)^{k}\rmd x\bigr)$, and
\[
    H_{1}^{j,k}(\mathbb{R})=\left\{  r\in H^{j,k}(\mathbb{R})\,:\,
        \|r\|_{\infty}<1\right\}  .
\]
The $H^{j,k}$ classes are very natural in that certain such classes are
preserved under the completely integrable flows for the NLS equation, the mKdV
equation, etc. On the other hand, it is of interest to see how far these can
be relaxed and still obtain a bijective map. While our class is not preserved
under KdV or mKdV, the mappings~$\mathit{\Phi}_\pm$ extend to the space%
\[
    \left(  L^{2}(\mathbb{R})\right)_{1}
        =\left\{  F\in L^{2}(\mathbb{R}) \,:\, \widehat{F}\in L^{\infty}(\mathbb{R}),\ \| \widehat{F}\|_{\infty}<1\right\}.
\]
This space is preserved under maps of the form $\widehat{F}\mapsto \rme^{\rmi s^{p}t}\widehat{F}$. \ In a separate paper \cite{BHP:2009}, we will use this
extension to construct Strichartz solutions to the defocusing cubic NLS equation and
study their asymptotic behavior.

We close this introduction with several examples of Miura potentials.

\begin{example}
Let $u$ be an even function that for $x>0$ equals $x^{-\alpha}\sin x^{\beta}$.
Assume that $\alpha>1$ and $\beta>\alpha+1$. Then $u$ belongs to~$X$ and the corresponding Miura potential $q=u^{\prime}+u^{2}$ is
of the form
\[
    q(x)=   \beta\,\mathrm{sign}\,(x) |x|^{\beta-\alpha-1}\cos|x|^{\beta}%
            +\tilde{q}(x)
\]
for some bounded function $\tilde{q}$. Thus $q$ is unbounded and oscillatory;
nevertheless, the corresponding Schr\"{o}dinger operator possesses only
absolutely continuous spectrum filling out the positive semi-axis and the
scattering and inverse scattering on such potential is well-defined.
\end{example}

\begin{example}\label{eq.coulomb}
Assume that $\phi\in C_0^\infty(\mathbb{R})$ is such that $\phi\equiv1$ on $(-1,1)$. Take a function $u(x) = \alpha\,\phi(x)\log|x|$ with $\alpha>0$. Then~$u\in X$; moreover, since the distributional derivative of $\log|x|$ is the distribution~${\mathrm{P.v.}}\,1/x$,   the corresponding Miura potential~$q$ is smooth outside the origin and has there a Coulomb-type singularity. See e.g.~\cite{BDL:2000,Kurasov:1996,FLM:1995} and the references therein for discussion and rigorous treatment of Schr\"odinger operators with Coulomb potentials.
\end{example}

\begin{example}(Frayer~\cite{Frayer:2008})
The Riccati representative $u=\alpha\chi_{[-1,1]}$, with $\alpha$ a nonzero real constant and $\chi_{\Delta}$ the indicator function of a set~$\Delta$, corresponds to the Miura potential
\[
    q=\alpha\delta(\cdot+1)-\alpha\delta(\cdot-1)+\alpha^{2}\chi_{\lbrack-1,1]},
\]
$\delta$ being the Dirac delta-function centered at the origin.
One can explicitly calculate the Jost solutions and hence the reflection
coefficients by matching solutions in the regions $x<-1$, $-1<x<1$, and $x>1$
at $x=-1$ and $x=+1$; however, the formula is rather involved. We omit details.
\end{example}

\begin{example}\label{ex.delta}
For the potential~$q(x)=\alpha\delta(x)$, with~$\alpha$ a nonzero real constant, the extremal solutions $\varphi_{\pm}$ of the zero-energy Schr\"{o}dinger equation (see Section~\ref{sec.miura}) are different and equal
\[
    \varphi_{+}(x)=%
        \cases{ 1 & for  $x>0$,\\
                1-\alpha x & for $x<0$,\\ }
\]
and
\[
    \varphi_{-}(x)=
        \cases{1+\alpha x & for  $x>0$,\\
               1 & for $x<0$,\\ }
\]
\it
respectively (see also~\cite[Appendix~A]{KPST:2005} and~\cite{Levitan:1979}). Therefore such~$q$ does not belong to the class of singular potentials treated in this paper.

For this potential, the reflection coefficient satisfies $r(0)=-1$. This is in
fact a generic situation even for potentials $q$ in the Faddeev--Marchenko
class $L^{1}\bigl(\mathbb{R},(1+|x|)\,dx\bigr)$, for which the standard scattering theory is
well understood. Therefore a generic Faddeev--Marchenko potential is not a
Miura potential in the sense of this paper.

As mentioned above, we shall extend the scattering theory developed here to a
wider class of potentials including delta functions, Faddeev--Marchenko
potentials and many other in the papers \cite{HMP:2008a} and \cite{HMP:2008b}.
\end{example}

This paper is organized as follows. In \S \ref{sec.pre} we introduce necessary notation and function spaces and recall some basic results about the Miura map. We review the direct scattering problem in \S \ref{sec.direct} and prove Theorem \ref{thm.1} in
\S \ref{sec.inverse}. In an appendix, we review the Wiener and
associated algebras used to analyze the direct and inverse scattering maps.


\section{Preliminaries}\label{sec.pre}

\subsection{Notation}\label{ssec.notat}

In what follows, $\mathrm{M}_{2}(\mathbb{C})$ will stand for the $2\times2$ complex matrices
with the Euclidean operator norm $|\,\cdot\,|$ and
\[
    \sigma_{3}=\left(\begin{array}[c]{cc}%
                        1 & 0\\
                        0 & -1
                    \end{array}\right)
\]
is the usual Pauli matrix. Upright capital Roman or Greek letters (e.g., $\mathrm{A}$, $\mathrm{I}$, $\mathrm{\Psi}$ etc) will usually denote elements of~$\mathrm{M}_{2}(\mathbb{C})$; in particular, $\mathrm{I}$ is the $2\times 2$ identity matrix.

Next, $L^1(\mathbb{R})$ and $L^2(\mathbb{R})$ are the standard Lebesgue function spaces on the real line~$\mathbb{R}$, and $\| \,\cdot\,\|_{1}$ and~$\|\,\cdot\,\|_{2}$ will stand for the respective $L^{1}$ and $L^{2}$-norms of scalar, vector, or matrix-valued functions on~$\mathbb{R}$. Recalling the Banach space $X$ (see (\ref{eq.X}) and (\ref{eq.X.norm})), we
define, for any $c\in\mathbb{R}$,
\[
    X_{c}^{+}=L^{1}(c,\infty)\cap L^{2}(c,\infty)
\]
and%
\[
    X_{c}^{-}=L^{1}(-\infty,c)\cap L^{2}(-\infty,c).%
\]

As usual, $C_0^\infty(\mathbb{R})$ denotes the linear space of all functions of compact support that are infinitely often differentiable. The Schwartz class~$S(\mathbb{R})$ consists of all infinitely often differentiable functions $f$ such that $x^kf^{(l)}$ are bounded on~$\mathbb{R}$ for every natural $k$ and~$l$. $H^1(\mathbb{R})$ is the space of functions in~$L^2(\mathbb{R})$ whose distributional derivative also belong to~$L^2(\mathbb{R})$, the norm being
\[
  \|f\|_{H^1} = \bigl(\|f\|^2_2+\|f'\|^2_2\bigr)^{1/2}.
\]
The space $H^{-1}(\mathbb{R})$ is a dual space to~$H^1(\mathbb{R})$, i.e., the space of all distributions that are bounded functionals on~$H^1(\mathbb{R})$. If $Y$ is any of the above spaces, then $Y_{\mathrm{loc}}$ is defined as the space of all functions (or distributions) $f$ such that $f\phi\in Y$ for every $\phi\in C_0^\infty(\mathbb{R})$.

Finally, we denote by $H_{\pm}^{2}(\mathbb{R})$ the Hardy spaces of functions on the real line having respective representations of the form%
\begin{eqnarray*}
    f_{+}(s) &  =\int_{0}^{\infty} \rme^{\rmi s\zeta}g_{+}(\zeta)\,\rmd\zeta,\\
    f_{-}(s) &  =\int_{-\infty}^{0}\rme^{\rmi s\zeta}g_{-}(\zeta)\,\rmd\zeta
\end{eqnarray*}
for $g_{+}\in L^{2}(0,\infty)$ and $g_{-}\in L^{2}(-\infty,0)$. Clearly,
$L^{2}(\mathbb{R})=H_{-}^{2}(\mathbb{R})\oplus H_{+}^{2}(\mathbb{R})$. For
$f\in L^{2}(\mathbb{R})$, we denote by ${\mathcal C}$ the Cauchy integral operator%
\[
    \left({\mathcal C}f\right)  (z)
        =\frac{1}{2\pi i}\int_{\mathbb{R}}\frac{1}{s-z}f(s)\,\rmd s,
            \qquad z \in {\mathbb C} \setminus {\mathbb R},
\]
and by ${\mathcal C}_{\pm}$ the operators%
\begin{equation}\label{eq.Cauchy}%
    \left({\mathcal C}_{\pm}f\right)  (s)=\lim_{\varepsilon\downarrow0}
            \left(\mathcal{C}f\right)(s\pm i\varepsilon),
    \qquad s \in {\mathbb R},
\end{equation}
where the limit is taken in~$L^{2}(\mathbb{R})$. The operators~${\mathcal C}_{+}$ and~$-{\mathcal C}_{-}$ are orthogonal projections onto~$H_{+}^{2}(\mathbb{R})$ and~$H_{-}^{2}(\mathbb{R})$ respectively, and ${\mathcal C}_{+}-{\mathcal C}_{-}=I$ as operators on~$L^{2}(\mathbb{R})$. We have the formulas%
\begin{equation}\label{eq.C+-}%
    \left({\mathcal C}_{\pm}f\right)(s)
        = \pm \mathcal{F}_+^{-1}\chi_{\pm}\mathcal{F}_+f,
\end{equation}
where $\chi_{+}$ (resp. $\chi_{-}$) is the indicator function of~$(0,\infty)$
(resp.\ of $(-\infty,0)$).

Both the direct and inverse scattering maps are obtained as solutions of integral
equations that naturally lead to series representations. To capture their
analyticity properties, we make the following rather nonstandard definition (see \cite{Dineen:1981,Nachbin:1969} and\cite[Appendix~A]{PT:1987} for discussion of analytic mappings between Banach spaces).

\begin{definition}\label{def.anal}
Let~$X$ be a Banach space of complex-valued functions, let~$U$ be an open subset of~$X$ containing along with every~$f$ its complex conjugate $\overline{f}$, and let~$\mathit{\Psi}$ be a nonlinear mapping from~$U$ into a Banach space~$Y$. We say that~$\mathit{\Psi}$ is analytic in variables~$f$ and~$\overline f$ if  $\mathit{\Psi}(f) =\mathit{\Psi}_2(f,\overline{f})$, where $\mathit{\Psi}_2: U \times U \rightarrow Y$ is analytic.
\end{definition}

\begin{remark}
\label{rem.real}If~$\mathit{\Psi}\,:\,X \to Y$ is analytic in $f$ and $\overline{f}$, then the
restriction of~$\mathit{\Psi}$ to the subspace~$X_{\mathbb{R}}$ of real-valued
functions is a real-analytic map from $X_{\mathbb{R}}$ to $Y$.
\end{remark}

We also note that if~$\mathit{\Psi}$ is analytic in~$f$ and $\overline{f}$ and $\phi$ is an analytic function of a complex variable, then the point-wise composition~$\phi \circ \mathit{\Psi}$ often turns out to be analytic in~$f$ and $\overline{f}$, see Lemma~\ref{lem.anal}.

\subsection{The Miura map and Riccati representatives}

\label{sec.miura}

In this subsection we briefly recall some results of \cite{KPST:2005} relating
the Miura map and the set of positive solutions of the Schr\"{o}dinger equation
at zero energy. We begin by considering distributional solutions of the
Schr\"{o}dinger equation with a real-valued potential~$q\in H_{\mathrm{loc}%
}^{-1}(\mathbb{R})$. For such $q$, we define the quadratic form~$\mathfrak{s}$
via~(\ref{eq.h}) and let%
\[
    \lambda_{0}(q)
        :=\inf\left\{  \mathfrak{s}(\varphi)\,:\,
        \varphi\in C_{0}^{\infty}(\mathbb{R}),\ \| \varphi\| _{2}=1\right\}.
\]
The quadratic form $\mathfrak{s}$ is nonnegative if $\lambda_{0}(q)\geq0$. A
function $y\in H_{\mathrm{loc}}^{1}(\mathbb{R})$ is a (\emph{distributional}\/) solution of $-y^{\prime\prime}+qy=0$ if for every $\varphi\in C_{0}^{\infty}(\mathbb{R})$, $\mathfrak{s}(y,\varphi)=0$, with $\mathfrak{s}(y,\varphi)$ denoting the value of the associated sesquilinear form. It is not difficult to see that any solution $y$ may only have isolated zeros where the solution $y$ changes sign;
in particular, if $y$ is nonnegative, it is either identically zero or
strictly positive. If $\lambda_{0}(q)\geq0$ and $y$ is a solution with either
$y\in L^{2}({\mathbb{R}}^{+})$ or $y\in L^{2}({\mathbb{R}}^{-})$, then $y$ has
a single sign or is identically zero: if $y$ has a zero, then there is a
half-line Dirichlet problem operator $-\rmd^{2}/\rmd x^{2}+q$ with a negative
eigenvalue, contradicting the assumption that $\lambda_{0}(q)\geq0$. Finally,
$\lambda_{0}(q)\geq0$ if and only if the set
\[ 
    \Pos(q) = \left\{  \psi\in H_{\mathrm{loc}}^{1}(\mathbb{R})
            \,:\, \psi>0,\ \psi(0)=1, \
            \forall\varphi\in C_{0}^{\infty}(\mathbb{R}) \quad \mathfrak{s}(\varphi,\psi)=0 \right\}
\]
of positive distributional solutions to the equation $-\psi^{\prime\prime}+q\psi=0$ is nonempty. If $\psi\in\Pos(q)$, then $q=u^{\prime}+u^{2}$
for $u=\psi^{\prime}/\psi$.

We initially consider the Miura map $B:\, u \mapsto u^{\prime}+u^{2}$ acting
from $L_{\mathrm{loc}}^{2}(\mathbb{R})$ into~$H_{\mathrm{loc}}^{-1}%
(\mathbb{R})$, by interpreting $u^{\prime}$ as the distribution derivative of~$u$. If~$q$ lies in the range~$\Ran B$ of $B$, then any function $u\in B^{-1}(q)$ is called a \emph{Riccati representative} of~$q$. The range of the Miura map consists of those $q\in H_{\mathrm{loc}}^{-1}(\mathbb{R})$, for which $\lambda_{0}(q)\geq0$, or, equivalently, for which the set
$\Pos(q)$ is nonempty. The set $\Ran B$ is nowhere dense in~$H_{\mathrm{loc}}^{-1}(\mathbb{R})$.

It is not difficult to see that the mappings
\[
    u\mapsto\exp\left(\int_{0}^{x}u(s)~\rmd s\right)
            \qquad\mathrm{and}\qquad
    \varphi\mapsto\frac{\varphi^{\prime}}{\varphi}%
\]
take $B^{-1}(q)$ into $\Pos(q)$ and vice versa, so that these
two sets are homeomorphic respectively as subsets of $L_{\mathrm{loc}}^{2}(\mathbb{R})$ and $H_{\mathrm{loc}}^{1}(\mathbb{R})$. Given any $y_{1}\in\Pos(q)$, any solution of~$-y^{\prime\prime}+qy=0$ with
$y(0)=1$ (whether positive or not) takes the form
\[
    y(x)=y_{1}(x)\left(  1+c\int_{0}^{x}y_{1}(s)^{-2}~\rmd s\right)
\]
for some $c$ (see \cite[Ch.~IX.2\,(ix)]{Hartman:1964}). Using this fact, it is not
difficult to see that the set $\Pos(q)$ is the set of convex
combinations of extremal solutions $\varphi_{+}$ and $\varphi_{-}$
characterized by the respective divergence of
 $\int_{0}^{\infty}\varphi_{+}^{-2}(s)\,\rmd s$ and $\int_{-\infty}^{0}\varphi_{-}^{-2}(s)\,\rmd s$.
As we will see, the extremal solutions~$\varphi_{+}$ and~$\varphi_{-}$ coincide up to constant factors with the usual right and left Jost solutions at zero energy
whenever the latter are defined. The extremal solutions $\varphi_{+}$ and
$\varphi_{-}$ either coincide or are linearly independent, so that $B^{-1}(q)$
(with the topology induced from~$L_{\mathrm{loc}}^{2}(\mathbb{R})$) is
homeomorphic either to a point or to a line segment.

Here we are concerned with the restriction of $B$ to $L^{2}(\mathbb{R})$ so
that the range of $B$ is contained in $H^{-1}(\mathbb{R})$. A distribution
$q\in H^{-1}(\mathbb{R})$ lies in the range of $B$ if and only if the
positivity condition $\lambda_{0}(q)\geq0$ holds and $q$ can be presented as
$f^{\prime}+g$ for $f\in L^{2}(\mathbb{R})$ and $g\in L^{1}(\mathbb{R})$. The
following result is a consequence of the proof of Lemma~4.1 in~
\cite{KPST:2005}.

\begin{lemma}\label{lem:uL2}
Suppose that $q\in H^{-1}(\mathbb{R})$ and that there exists a
Riccati representative~$u$ of~$q$ with~$u\in L^{2}({\mathbb{R}}^{+})$. Then
the same is true for every Riccati representative of $q$. A similar statement
holds with $L^{2}({\mathbb{R}}^{+})$ replaced by $L^{2}({\mathbb{R}}^{-})$.
\end{lemma}

The following simple uniqueness result is very important for the present paper.

\begin{lemma}\label{lem:u-uniq}
Suppose that $q\in H^{-1}(\mathbb{R})$ and that~$u_{1}$ and~$u_{2}$ are
Riccati representatives of~$q$ belonging to~$X_0^+$. Then $u_{1}=u_{2}$. Similarly, if~$u_{1}$ and~$u_{2}$ are Riccati representatives of~$q$ belonging to~$X_0^-$, then~$u_{1}=u_{2}$.
\end{lemma}

\begin{proof}
If $u\in X_0^+=L^{1}({\mathbb{R}}^{+})\cap L^{2}({\mathbb{R}}^{+})$ belongs to~$B^{-1}(q)$, then
$\varphi(x):=\exp\left(\int_{0}^{x}u(s)\,\rmd s\right)$ is bounded above and
below by fixed constants for $x>0$, so $\int_{0}^{\infty}\varphi(s)^{-2}\,\rmd s$
diverges. Since $\varphi$ solves the equation $-y'' + qy=0$, we conclude that $\varphi=\varphi_{+}$ is an extremal positive solution, and $u=\varphi_{+}^{\prime}/\varphi_{+}$.
Hence $u_{1}=u_{2}=\varphi_{+}^{\prime}/\varphi_{+}$. The other proof is similar.
\end{proof}

Recall that $X=L^{1}(\mathbb{R})\cap L^{2}(\mathbb{R})$. We immediately obtain
from Lemma~\ref{lem:u-uniq} the following uniqueness result:

\begin{corollary}
\label{cor:u-uniq} Every $q\in H^{-1}({\mathbb{R}})$ has at most one Riccati
representative $u$ belonging to~$X$.
\end{corollary}

\section{Direct Scattering}\label{sec.direct}

Although direct scattering for the ZS-AKNS\ system is well-understood (see, for
example, Beals and Coifman~\cite{BC:1984}, Zhou~\cite{Zhou:1998}, and the
monographs of Faddeev and Takhtajan~\cite{FT:2007} and Beals, Deift, and Tomei~\cite{BDT:1988}), we give a brief synopsis here. We will study the direct
scattering map, derive representation formulas for the scattering solutions,
and finally derive the Gelfand--Levitan--Marchenko equations and the reconstruction formula.

It is useful to factor out the leading behavior of solutions to~(\ref{eq.AKNS}) with $\mathrm{Q}$ of~\eref{eq.Q} by setting
\[
    \Psi(x,z)=\mathrm{M}(x,z)\,\exp(\rmi zx\sigma), %
\]
where $\mathrm{M}$ takes values in $2\times2$ complex matrices, so that%
\begin{equation}
    \mathrm{M}'_{x}=\rmi z\, \ad_\sigma (\mathrm{M}) + \mathrm{Q}(x)\mathrm{M}, \label{eq.m}%
\end{equation}
where%
\begin{equation*}
    \ad_{\mathrm{A}}(\mathrm{B}) := [\mathrm{A},\mathrm{B}] = \mathrm{A}\mathrm{B}-\mathrm{B}\mathrm{A}. 
\end{equation*}
To study scattering asymptotics, we remove the linear term from~(\ref{eq.m})
by setting%
\begin{equation} \label{eq.mn}%
    \mathrm{M}(x,z)=\exp(\rmi zx\,\ad_\sigma)\, \mathrm{N}(x,z),
\end{equation}
where%
\begin{equation*}  
    \exp(\rmi t\,\ad_\sigma)
         \left(\begin{array}[c]{cc}
                a & b\\
                c & d
                \end{array}\right)
        =\left(\begin{array}[c]{cc}
                a & e^{it}b\\
                e^{-it}c & d
                \end{array}\right)
\end{equation*}
is an automorphism of the Lie group
\[
    SU(1,1):=\left\{  \mathrm{Y}\in \mathrm{M}_{2}(\mathbb{C})\,:\, \mathrm{Y}^{\ast}\sigma \mathrm{Y}=\sigma,\quad\det \mathrm{Y}=1\right\}.
\]
Then $\mathrm{N}(x,z)$ obeys the differential equation
\begin{equation}
    \mathrm{N}'_{x}(x,z)=\mathrm{G}(x,z)\,\mathrm{N}(x,z) \label{eq.n}%
\end{equation}
where%
\begin{eqnarray}\label{eq.G}
    \mathrm{G}(x,z)    :=\exp(-\rmi zx\,\ad_\sigma)\, \mathrm{Q}(x)
              =\left(\begin{array}[c]{cc}%
                0 & \rme^{-\rmi zx}u(x)\\
                \rme^{\rmi zx}\overline{u(x)} & 0
                    \end{array}\right).
\end{eqnarray}
Since $\mathrm{G}$ is traceless and $\sigma \mathrm{G}+\mathrm{G}\sigma=0$, it is not difficult to see
that~$\det \mathrm{N}$ and $\mathrm{N}^{\ast}\sigma \mathrm{N}$ are independent of~$x$ for any solution
of~(\ref{eq.n}).

Now consider the singular initial value problem%
\begin{eqnarray}\label{eq.n.sivp}
    \eqalign{   \mathrm{N}'_{x}(x,z)     &=\mathrm{G}(x,z)\,\mathrm{N}(x,z) \cr
                \mathrm{N}(\infty,z)     &=\mathrm{I}.}
\end{eqnarray}
Reformulating this problem as an integral equation%
\begin{equation} \label{eq.n.ie}
    \mathrm{N}(x,z)=\mathrm{I}+\int_{\infty}^{x}\mathrm{G}(y,z)\mathrm{N}(y,z)\,\rmd y,
\end{equation}
it is not difficult to see that if $u\in L^{1}({\mathbb R})$ and $z=s$ is real, there exists a unique solution which is a continuous curve in $SU(1,1)$ with a well-defined limit as $x\rightarrow-\infty$. Tracing through the definitions and comparing
with (\ref{eq.Psipm.A}), we see that
\[
    \mathrm{A}(s)=\mathrm{N}(-\infty,s).
\]
Since $\mathrm{N}(-\infty,s)\in SU(1,1)$, we may write%
\begin{equation}\label{eq.N.-infty}
    \mathrm{N}(-\infty,s)=\left(\begin{array}[c]{cc}
                    a(s) & \overline{b}(s)\\
                    b(s) & \overline{a}(s)
                        \end{array}\right).
\end{equation}
In what follows, we will study the direct scattering map by studying the
singular initial value problem~(\ref{eq.n.sivp}). We will then return to~(\ref{eq.m}) in order to study the scattering solutions and their analyticity
properties. From these properties we can deduce the Gelfand--Levitan--Marchenko
equations and reconstruction formulas for~$u$.

\subsection{The Direct Scattering Map}

The basic objects of direct scattering are the reflection coefficients
(\ref{eq.r-}) and (\ref{eq.r+}) computed from the asymptotic behavior of
solutions to (\ref{eq.AKNS}). It is simplest to study the map
\[
    u\mapsto    \left( \begin{array}[c]{cc}%
                a(s) & \overline{b}(s)\\
                b(s) & \overline{a}(s)
                \end{array} \right)
\]
using~(\ref{eq.n.sivp}) and (\ref{eq.N.-infty}) and deduce properties of the reflection coefficients from
those of~$a$ and~$b$.

The first result is:

\begin{lemma}
\label{lemma.ab}Suppose that $u\in X$. Then the representations%
\begin{eqnarray}
    a(s)  &  =1+\int_{0}^{\infty}
        \rme^{\rmi s\zeta}A(\zeta)\,\rmd\zeta,\label{eq.a.rep}\\
    b(s)  &  =\int_{-\infty}^{\infty}
        \rme^{\rmi s\zeta}B(\zeta)\,\rmd\zeta\label{eq.b.rep}%
\end{eqnarray}
hold, where $A,B\in X$ with%
\begin{eqnarray}
    \| A\| _{1},\| B\| _{1}
        &  \leq\exp\left(\| u\| _{1}\right)-1, \label{eq.AB.L1}\\
    \| A\| _{2},\| B\| _{2}
        &  \leq\exp\left(\| u\| _{1}\right) \| u\| _{2}. \label{eq.AB.L2}%
\end{eqnarray}
Moreover, the maps $u\mapsto A$ and $u\mapsto B$ are analytic in $u$ and
$\overline{u}$ in the sense of De\-fi\-ni\-ti\-on~\ref{def.anal}.
\end{lemma}

\begin{proof}
To obtain the representations (\ref{eq.a.rep})--(\ref{eq.b.rep}) and the
estimates (\ref{eq.AB.L1})--(\ref{eq.AB.L2}), we study the Volterra series for
(\ref{eq.n.ie}). A straightforward computation with (\ref{eq.n.ie}) using
(\ref{eq.G}) shows that the entries~$n_{11}$ and $n_{21}$ of~$\mathrm{N}$ obey the system of integral equations%
\begin{eqnarray}
    n_{11}(x,z)  &  =1+\int_{\infty}^{x}
        \rme^{-\rmi zy}u(y)n_{21}(y,z)\,\rmd y, \label{eq.n11}\\
    n_{21}(x,z)  &  =\int_{\infty}^{x}
        \rme^{\rmi zy}\overline{u}(y)n_{11}(y,z)\,\rmd y. \label{eq.n21}%
\end{eqnarray}
We can iterate equations (\ref{eq.n11})--(\ref{eq.n21}) to obtain Volterra
series representations, and then obtain the required formulas for $a$ and $b$
since, by (\ref{eq.N.-infty}), we have~$a=n_{11}(-\infty,\,\cdot\,)$ and~$b=n_{21}(-\infty,\,\cdot\,)$. First, we have%
\begin{eqnarray}
    n_{11}(x,s)  &  =1+\sum_{n=1}^{\infty}T_{2n}(x,s),\label{eq.n11.series}\\
    n_{21}(x,s)  &  =\sum_{n=0}^{\infty}T_{2n+1}(x,s). \label{eq.n21.series}%
\end{eqnarray}
Here%
\[
    T_{n}(x,s):=\int_{x\leq y_{1}\leq\cdots\leq y_{n}}
        \exp(\rmi z\varphi_{n}(\rmy))U_{n}(\rmy)~\rmd\rmy,
\]
where~$\rmy:=(y_1,\dots,y_n)$,
\begin{eqnarray*}
    \varphi_{n}(\rmy):=\sum_{j=0}^{n-1}(-1)^{j}y_{n-j}, \label{eq.phin} \\%
    U_{2n}(\rmy):=u(y_{1})\overline{u}(y_{2})\cdots u(y_{2n-1})\overline{u}(y_{2n}),
\end{eqnarray*}
and%
\begin{equation*}
    U_{2n+1}(\rmy)
        :=-\overline{u}(y_{1})u(y_{2})\overline{u}(y_{2})
            \cdots u(y_{2n})\overline{u}(y_{2n+1}). 
\end{equation*}
Note that $\varphi_{2n}(\rmy)\geq0$ if $y_{1}\leq\cdots\leq y_{2n}$, whereas
$\varphi_{2n-1}(\rmy)\geq x$ for $\rmy$ in the range of integration for~$T_{2n-1}$. We may write%
\begin{eqnarray*}
    T_{2n}(x,s)
        &  =\int_{0}^{\infty}\rme^{\rmi s\zeta} A_{2n}(x,\zeta)\,\rmd\zeta,
                    \label{eq.T2n}\\
    T_{2n+1}(x,s)
        &  =\int_{x}^{\infty}\rme^{\rmi s\zeta} B_{2n+1}(x,\zeta)\,\rmd\zeta,
                    \label{eq.T2n+1}%
\end{eqnarray*}
where%
\begin{eqnarray*}
    A_{2n}(x,\zeta)  &  :=\int_{\overset{\varphi_{2n}(\rmy)=\zeta}
        {\underset{x\leq y_{1}\leq\cdots\leq y_{2n}}\null}}
            U_{2n}(\rmy)\,\rmd S_{2n},\label{eq.,A2n}\\
    B_{2n+1}(x,\zeta)  &  :=\int_{\overset{\varphi_{2n+1}(\rmy)=\zeta}
        {\underset{x\leq y_{1}\leq\cdots\leq y_{2n+1}}\null}}
            U_{2n+1}(\rmy)\,\rmd S_{2n+1} \label{eq.B2n+1}%
\end{eqnarray*}
are multilinear functions of $u$ and $\overline{u}$. Here $\rmd S_{n}$ is the surface
measure on the hypersurface~$\left\{\rmy\in\mathbb{R}^{n}\,:\,
        \varphi_{n}(\rmy)=\zeta\right\}$.
Using the fact that $\rmd \rmy=\rmd S_{n}\rmd\zeta$ and letting
$\eta$ and $\gamma$ be the bounded monotone functions%
\begin{eqnarray*}
    \eta(x)  &  :=\int_{x}^{\infty}| u(t)| \,\rmd t,\\
    \gamma(x)  &  :=\left(\int_{x}^{\infty}| u(t)|^{2}\rmd t\right)^{1/2},
\end{eqnarray*}
we can easily derive the estimates
\begin{eqnarray*}
    \| A_{2n}(x,\,\cdot\,)\|_{1} &  \leq\frac{\eta(x)^{2n}}{(2n)!},\\
    \| B_{2n+1}(x,\,\cdot\,)\| _{1}  & \leq\frac{\eta(x)^{2n+1}}{(2n+1)!},\\
    \| A_{2n}(x,\,\cdot\,)\| _{2}  &  \leq\frac{\eta(x)^{2n-1}}{(2n-1)!}
                \gamma(x),\\
    \| B_{2n+1}(x,\,\cdot\,)\| _{2}  &  \leq\frac{\eta(x)^{2n}}{(2n)!}
                \gamma(x).
\end{eqnarray*}

These estimates show that the Volterra series (\ref{eq.n11.series})--(\ref{eq.n21.series}) converge and define bounded absolutely continuous functions of~$x$ having finite limits as~$x\rightarrow-\infty$ for each~$s$. By taking limits in~(\ref{eq.n11.series})--(\ref{eq.n21.series}), it is then easy to deduce that~$a$ and~$b$ have the series representations
\begin{eqnarray*}
    a(s)  &  =1+\sum_{n=1}^{\infty}T_{2n}(s),\\
    b(s)  &  =\sum_{n=1}^{\infty}T_{2n-1}(s),
\end{eqnarray*}
where
\[
    T_{n}(s):=\lim_{x\rightarrow-\infty}T_{n}(x,s)
\]
is a bounded continuous function of~$s$ for each~$n$. Setting~$A_{2n}%
(\zeta)=A_{2n}(-\infty,\zeta)$ and $B_{2n-1}(\zeta)=B_{2n-1}(-\infty,\zeta)$
and defining%
\begin{eqnarray*}
    A(\zeta)  &  :=\sum_{n=1}^{\infty}A_{2n}(\zeta),\\
    B(\zeta)  &  :=\sum_{n=0}^{\infty}B_{2n+1}(\zeta),
\end{eqnarray*}
we obtain the required representations. The continuity follows from the
estimates%
\begin{eqnarray*}
    \| A_{2n}-\widetilde{A}_{2n}\| _{p}
        &  \leq\frac{C^{2n-1}}{(2n-1)!}\| u-\widetilde{u}\|_{p},\\
    \| B_{2n+1}-\widetilde{B}_{2n+1}\| _{p}
        &  \leq\frac{C^{2n}}{(2n)!}\| u-\widetilde{u}\|_{p}%
\end{eqnarray*}
for $p=1,2$ where $C:=\| u\| _{1}+\|\widetilde{u}\|_{1}$.
\end{proof}

Since $\mathrm{A}(s)\in SU(1,1)$ for every $s\in\mathbb R$, we have the identity
\begin{equation}\label{eq.a-b}
    |a(s)|^2 - |b(s)|^2=1,
\end{equation}
and thus $a$ does not vanish on the real line. Moreover, the integral representation~\eref{eq.a.rep} shows that the function~$a$ is the boundary value of a function of the complex variable~$z$ (also denoted by~$a$) that is defined in the closed upper half-plane $\overline{\mathbb C^+}$ by the formula
\[
    a(z) := 1+\int_{0}^{\infty}
        \rme^{\rmi z\zeta}A(\zeta)\,\rmd\zeta.
\]
The function~$a$ so defined is continuous in~$\overline{\mathbb C^+}$, analytic in~${\mathbb C}^+$, and has no zeros there (see, e.g.,~\cite{AKNS:1974}, \cite[Sect.~3]{DZ:2003}).

\begin{remark}\label{rem.a}
Since $a$ does not vanish on the real line and tends to $1$ at infinity in view of~\eref{eq.a.rep} and the Riemann--Lebesgue lemma, it follows from Lemma~\ref{lem:WL} that $a$ is an
invertible element of the Banach algebra $1\dotplus\widehat{X}$~(see~\ref{app.Wiener}). The function $a^{-1}$ can be extended to the upper half-plane~${\mathbb C}^+$ as an analytic function~$1/a(z)$; thus $a^{-1}$ belongs to the Hardy space~$H^2_+$ and
\[
    a^{-1}(z) = 1+\int_{0}^{\infty}C(\zeta)\rme^{\rmi z\zeta}\,\rmd\zeta
\]
for a function $C\in X^+_0$. Moreover, $C$ is analytic in $u$ and $\overline{u}$
in the sense of Definition~\ref{def.anal}, as follows from Lemma~\ref{lem.anal} by setting $\phi(z):=1/z$.
\end{remark}

Recalling the formulas (\ref{eq.r-}) and (\ref{eq.r+}), the constraint
$\left\vert a(s)\right\vert ^{2}-\left\vert b(s)\right\vert ^{2}=1$, and the
bound $\| a\| _{\infty}\leq1+\| A\| _{1}$,
it is easy to see that the reflection coefficients~$r_{+}$ and~$r_{-}$ are continuous with~$\| r_{\pm}\| _{\infty}<1$. By Remark~\ref{rem.a},
$r_{\pm}$ are also elements of the Banach algebra~$\widehat{X}$ (i.e., Fourier transforms of functions in~$X$) which
are analytic in $u$ and $\overline{u}$. Thus:

\begin{lemma}
\label{lemma.direct}The maps $\mathcal{F}_{+}\circ\mathcal{S}_{+}$ and
$\mathcal{F}_{-}\circ\mathcal{S}_{-}$ are continuous maps of $X$ into $X_{1}$,
analytic in the sense of Definition~\ref{def.anal}.
\end{lemma}

It will be important to formulate the exact relationship between~$r_{+}$ and~$r_{-}$. It follows from the definitions~\eref{eq.r-}--\eref{eq.r+} that
\[
    r_{-}(s)=-\overline{r_{+}(s)}\,\frac{\overline{a}(s)}{a(s)}.
\]
The relation~\eref{eq.a-b} yields for~$s\in \mathbb R$ the identity
\[
    |a(s)|^{-2} = 1 - |r(s)|^2,
\]
where $r=r_-$ or $r=r_+$, so that $|a|$ is determined by~$r_\pm$ on the real line~$\mathbb R$. Since the function~$a$ is bounded and continuous in~$\overline{\mathbb C^+}$, analytic in~$\mathbb C^+$, and has no zeros in~$\overline{\mathbb C^+}$, the function $\log a$ is well defined in~$\overline{\mathbb C^+}$ and analytic in~${\mathbb C}^+$. Observe that
\[
    \myRe \log a(s) = \log |a(s)| = - {1\over2} \log (1-|r(s)|^2);
\]
thus we can use the Schwarz formula to reconstruct the function~$\log a$ from its real part on~$\mathbb R$. Explicitly, we get%
\[
    a(z)
        =\exp\Bigl[  \frac{1}{2\pi \rmi}\int_{-\infty}^{\infty}
            \frac{1}{\zeta-z}
            \log\left(  1-\left\vert r(\zeta)\right\vert^{2}\right)\,\rmd\zeta\Bigr],
\]
where $r=r_{+}$ or $r_{-}$; the boundary value is
\begin{equation}\label{eq.a}
    a(s)
    =\exp \left\{ {\mathcal C}_+
        \left[ \log (  1-| r(\,\cdot\,)|^{2} ) \right](s)\right\},
\end{equation}
where~${\mathcal C}_{+}$ is the Cauchy operator~(\ref{eq.Cauchy}), which is bounded
from $L^{2}(\mathbb R)$ to itself. We also observe that the problem of reconstruction of the function $a$ in $\mathbb C^+$ from the values of $|a|$ on the real line could be reduced to the classical Riemann--Hilbert problem, cf.~\cite{DZ:2003,Zhou:1998}.

Let us denote by $\widehat{X}_{1}$ the image of $X_{1}$ under the Fourier
transform~\eref{eq.Fourier}. If $r\in\widehat{X}_{1}$, it follows from
Banach algebra properties of $\widehat{X}$ (see~\ref{app.Wiener}) and Lemma \ref{lem:WL} that,
also, $\log\left(  1-\left\vert r(\zeta)\right\vert ^{2}\right)
\in\widehat{X}$. By Lemma~\ref{lem.anal}, $\log\left(  1-\left\vert
r(\zeta)\right\vert ^{2}\right)$ is analytic in $r$ and $\overline{r}$ in
the sense of Definition \ref{def.anal}. From (\ref{eq.C+-}) we deduce that the function
\[
    g (s)= {\mathcal C}_{+}
        \left[ \log (  1-| r(\,\cdot\,)|^{2} ) \right](s)
\]
also belongs to~$\widehat{X}$ and is analytic in $r$ and $\overline{r}$ in
the sense of Definition \ref{def.anal}. It now follows from Lemma \ref{lem:WL}
that
\[
    a(s)-1 = \exp[g(s)]-1
\]
belongs to $\widehat{X}$ as well and is also analytic in the same
sense. Hence, $a$ of~\eref{eq.a} is an invertible element of $1\dotplus\widehat{X}$
whenever $r\in$ $\widehat{X}_{1}$.

Now we define a nonlinear map
\[
    \mathcal{I}:\widehat{X}_{1}\rightarrow\widehat{X}_{1}%
\]
by%
\[
    (\mathcal{I}r)(s)=-\overline{r(s)}~\frac{\overline{a}(s)}{a(s)},
\]
where $a$ is given by~(\ref{eq.a}). Note that in the Schr\"{o}dinger case, where (\ref{eq.r.real}) holds and,
also, $a(-s)=\overline{a(s)}$, the map~$\mathcal{I}$ is given by
\[
    r\mapsto-r(-s)\frac{a(-s)}{a(s)}%
\]
(cf. for example \cite{DT:1979}, \S 2.2, Theorem 1).

\begin{lemma}
The map $\mathcal{I}$ is a continuous involution of $\widehat{X}_{1}$,
analytic in $r$ and $\overline{r}$ in the sense of Definition \ref{def.anal}.
\end{lemma}

\begin{proof}
The map $r\mapsto a$ defined by the solution of the Schwarz problem above is continuous from $\widehat{X}_{1}$ to $1\dotplus\widehat{X}$ and
$a\mapsto a/\overline{a}$ is continuous from $1\dotplus\widehat{X}$
to itself. Multiplication by $r$ is continuous from~$1\dotplus\widehat{X}$ to~$\widehat{X}$, so we conclude that
$\mathcal{I}$ is continuous as claimed. To see that $\mathcal{I}$ is an
involution, we note that $\left\vert (\mathcal{I}r)(s)\right\vert =\left\vert
r(s)\right\vert $ and deduce from (\ref{eq.a}) that $\mathcal{I}^{2}r=r$.
Finally, the analyticity follows from analyticity of the map $r\mapsto
a$ and invertibility of~$a$ in $1 \dotplus \widehat{X}$.
\end{proof}

It now follows that the identities%
\begin{equation*}
    \mathcal{S}_{-}=\mathcal{I}\circ\mathcal{S}_{+} \label{eq.R+.to.R-}%
\end{equation*}
and%
\begin{equation*}
    \mathcal{S}_{+}=\mathcal{I}\circ\mathcal{S}_{-} 
\end{equation*}
hold.

\subsection{Scattering Solutions}

Recalling the scattering solutions $\Psi_{\pm}$ of (\ref{eq.AKNS}) with
\[
    \lim_{x\rightarrow\pm\infty}|\Psi_{\pm}(x,z)-\exp(\rmi zx\sigma)| =0,
\]
let us write
\[
\Psi_{\pm}(x,z)=\mathrm{M}_{\pm}(x,z)\exp(\rmi zx\sigma),%
\]
so that $\mathrm{M}_{\pm}$ solve (\ref{eq.m}) with
\[
\lim_{x\rightarrow\pm\infty} | \mathrm{M}_{\pm}(x,z)-\mathrm{I}| =0.
\]
These singular initial value problems can be reformulated as Volterra integral
equations%
\begin{equation}
\mathrm{M}_{\pm}(x,z)=\mathrm{I}+\int_{\pm\infty}^{x}\exp\bigl(\rmi z(x-y)\,\ad_\sigma\bigr)
    \left[  \mathrm{Q}(y)\mathrm{M}(y,z)\right]\,\rmd y. \label{eq.m.int}%
\end{equation}
These equations have a unique absolutely continuous solution for each real $z$
provided that $u\in L^{1}(\mathbb{R})$. \ The map $x\mapsto$ $\mathrm{M}_{\pm}(x,z)$ is
a continuous curve in $SU(1,1)$ for each fixed $z.$ Letting $m_{ij}^{\pm}$
denote the entries of $\mathrm{M}_{\pm}$, it is not difficult to see that
\begin{eqnarray}
    m_{11}^{+}(x,z)
        &  =1+\int_{+\infty}^{x} u(y)m_{21}^{+}(y,z)\,\rmd y,\label{eq.m11}\\
    m_{21}^{+}(x,z)
        &  =\int_{+\infty}^{x}\rme^{-\rmi z(x-y)}\overline{u}(y)m_{11}^{+}(y,z)\,\rmd y \label{eq.m21}%
\end{eqnarray}
(and similar formulas with all $+$ signs replaced by $-$ signs), while
\begin{eqnarray}
m_{12}^{+}(x,z)  &  =\overline{m_{21}^{+}(x,z)},\label{eq.m12.sym}\\
m_{22}^{+}(x,z)  &  =\overline{m_{11}^{+}(x,z)} \label{eq.m22.sym}%
\end{eqnarray}
(and similar formulas with $+$ replaced by $-$) since $\mathrm{M}_{\pm}$ take values
in $SU(1,1)$. Thus the matrix-valued function~$\mathrm{M}_+$ is determined by solutions of the system of two integral equations (\ref{eq.m11})--(\ref{eq.m21}), and ~$\mathrm{M}_-$ is determined by solutions of a similar system~for $m_{11}^{-}$ and $m_{21}^-$.

The integral equations (\ref{eq.m11})--(\ref{eq.m21}) can be analyzed in the
same way as equations (\ref{eq.n11})--(\ref{eq.n21}) in the proof of Lemma~\ref{lemma.ab}, or we can note from (\ref{eq.mn}) that%
\begin{eqnarray}
m_{11}^{+}(x,z)  =               &n_{11}^{+}(x,z), \label{eq.mn.11+}\\
m_{21}^{+}(x,z)  =\rme^{-\rmi zx}&n_{21}^{+}(x,z). \label{eq.mn.21+}%
\end{eqnarray}
It follows directly from the representations (\ref{eq.n11.series})--(\ref{eq.n21.series}) already established for $n_{11}^{+}$ and $n_{21}^{+}$
that%
\begin{eqnarray*}
    m^+_{11}(x,z)   = 1 + &\int_{0}^{\infty}
        \rme^{\rmi z\zeta}\Gamma_{11}^{+}(x,\zeta)\,\rmd\zeta,\\
    m^+_{21}(x,z)   =     &\int_{0}^{\infty}
        \rme^{\rmi z\zeta}\Gamma_{21}^{+}(x,\zeta)\,\rmd\zeta,
\end{eqnarray*}
where the estimates%
\begin{eqnarray*}
\| \Gamma_{11}^{+}(x,\,\cdot\,)\| _{1}
    &  \leq\exp\left[\eta(x)\right] -1 ,\\
\| \Gamma_{11}^{+}(x,\,\cdot\,)\| _{2}
    &  \leq\exp\left[\eta(x)\right]  \gamma(x)
\end{eqnarray*}
hold (note that in the series for $m_{21}^{+}$, the factor $\exp(-\rmi zx)$ in
(\ref{eq.mn.21+}) changes the phase function to $\varphi_{2n+1}(\rmy)-x$; this
phase function is nonnegative for any $x$).

The explicit multilinear expressions for $\Gamma_{11}^{+}(x,\,\cdot\,)$ and $\Gamma
_{21}^{+}(x,\,\cdot\,)$ show that they are analytic functions of $u$ and $\overline{u}$ in the sense of Definition~\ref{def.anal}, taking values in the space~$X_0^+$ for each real~$x$. Using the symmetry properties (\ref{eq.m12.sym})--(\ref{eq.m22.sym}), we recover the full matrix-valued function $\Gamma_{+}$. We can make an analogous analysis for $\mathrm{M}_{-}$, and obtain the following representations. In what follows, $X_0^\pm \otimes \mathrm{M}_2(\mathbb{C})$ denotes the space of matrix-valued functions over $\mathbb{R}^\pm$ with entries belonging to~$X_0^\pm$.

\begin{proposition}\label{prop.Gamma.rep}
Suppose that $u\in X$ and that $\mathrm{M}_{\pm}$
are the unique solutions of the integral equations (\ref{eq.m.int}). Then there
exist continuous functions
\[
    {\mathbb R} \ni x \mapsto \Gamma_+(x,\,\cdot\,)
            \in     X_0^+ \otimes \mathrm{M}_2(\mathbb{C})
\]
and
\[
 {\mathbb R} \ni x \mapsto \Gamma_-(x,\,\cdot\,)
            \in   X_0^- \otimes \mathrm{M}_2(\mathbb{C})
\]
with
\begin{equation*}
    \Gamma_{+}(x,\zeta)=\left(\begin{array}[c]{cc}%
        \Gamma_{11}^{+}(x,\zeta) & \overline{\Gamma_{21}^{+}(x,\zeta)}\\
        \Gamma_{21}^{+}(x,\zeta) & \overline{\Gamma_{11}^{+}(x,\zeta)}%
        \end{array}\right)  \label{eq.Gamma+.sym}%
\end{equation*}
and%
\begin{equation*}
    \Gamma_{-}(x,\zeta)=\left(\begin{array}[c]{cc}%
        \Gamma_{11}^{-}(x,\zeta) & \overline{\Gamma_{21}^{-}(x,\zeta)}\\
        \Gamma_{21}^{-}(x,\zeta) & \overline{\Gamma_{11}^{-}(x,\zeta)}%
        \end{array}\right), \label{eq.Gamma-.sym}%
\end{equation*}
so that the representations%
\begin{eqnarray}
    \mathrm{M}_{+}(x,z)  &  =\mathrm{I} + \int_{0}^{\infty}\Gamma_{+}(x,\zeta)
            \exp(\rmi z\zeta\sigma_{3})\rmd\zeta,\label{eq.m+.rep}\\
    \mathrm{M}_{-}(x,z)  &  =\mathrm{I} + \int_{-\infty}^{0}\Gamma_{-}(x,\zeta)
            \exp(\rmi z\zeta\sigma_{3})\rmd\zeta \label{eq.m-.rep}%
\end{eqnarray}
hold. Moreover, the estimates%
\begin{eqnarray*}
    \| \Gamma_{+}(x,\,\cdot\,)\| _{1}  &  \leq\exp[\eta_{+}(x)]-1,\\
    \| \Gamma_{-}(x,\,\cdot\,)\| _{1}  &  \leq\exp[\eta_{-}(x)]-1,\\
    \| \Gamma_{+}(x,\,\cdot\,)\| _{2}  &  \leq\exp[\eta_{+}(x)]
                \gamma_{+}(x),\\
    \| \Gamma_{-}(x,\,\cdot\,)\| _{2}  &  \leq\exp[\eta_{-}(x)]
                \gamma_{-}(x)
\end{eqnarray*}
hold, where
\begin{eqnarray*}
    \eta_{\pm}(x)  &  =\pm\int_{x}^{\pm\infty} | u(s) |\,\rmd s,\\
    \gamma_{\pm}(x)  &  =\Bigl(  \pm\int_{x}^{\pm\infty}
        | u(s)|^{2}\,\rmd s\Bigr)^{1/2}.
\end{eqnarray*}
Finally, for every fixed~$x\in {\mathbb R}$,
the maps~$u \mapsto \Gamma_\pm(x,\,\cdot\,)$ from~$X$ into~$X_0^\pm \otimes \mathrm{M}_2(\mathbb{C})$
are analytic in the sense of Definition~\ref{def.anal}.
\end{proposition}

\begin{remark}\label{rem.dependence}
The function $\Gamma_{+}(x,\zeta)$ depends only on
values of $u(y)$ with $y>x$, while the function $\Gamma_{-}(x,\zeta)$ depends
only on values of $u(y)$ with $y<x$. We will derive integral equations for
$\Gamma_{\pm}$ given the functions~$F_{\pm}$ of~(\ref{eq.F+})--(\ref{eq.F-}); the equation for $\Gamma_{+}$ will provide a
stable reconstruction of $u(x)$ on any half-line $\left(c,\infty\right)$,
while the equation for $\Gamma_{-}$ will provide a stable reconstruction for
any half-line $(-\infty,c)$.
\end{remark}

If we write $\mathrm{M}_{+}=\left(  m_{1}^{+},m_{2}^{+}\right)  $, where $m_{1}^{+}$ and
$m_{2}^{+}$ are column vectors, and similarly write $\mathrm{M}_{-}=\left(  m_{1}%
^{-},m_{2}^{-}\right)  $, then $\left(  m_{1}^{+},m_{2}^{-}\right)  $ extends
to an analytic matrix-valued function of $z$ in $\myIm z>0$,
while $(m_{1}^{-},m_{2}^{+})$ extends to an analytic matrix-valued function of
$z$ in $\myIm z<0$. \ Indeed, as an immediate corollary of
Proposition \ref{prop.Gamma.rep}, we have:

\begin{proposition}\label{prop.Hardy.m}%
For every fixed $x\in\mathbb{R}$, the functions
\[
      m_1^+(x, \,\cdot\,) - \left(\begin{array}[c]{c}1\\0\end{array}\right),	 \qquad
      m_2^-(x, \,\cdot\,) - \left(\begin{array}[c]{c}0\\1\end{array}\right)
\]
belong to the Hardy space~$H_+^2(\mathbb{R},\mathbb{C}^2)$.
\end{proposition}

These Hardy space properties are the starting point for the derivation of the
Gelfand--Levitan--Marchenko equations.

\subsection{The\ Gelfand--Levitan--Marchenko equations}

We now derive integral equations for $\Gamma_{\pm}$ in terms of $F_{\pm}$ and
also derive a reconstruction formula for $u$ from $\Gamma_{\pm}$. In the next
section we will solve these equations and prove consistency of the two
reconstructions. Given functions~$F_\pm$ of~(\ref{eq.F+})--(\ref{eq.F-}), let us define%
\begin{eqnarray*}
    \Omega_{-}(x)  &  :=\left(\begin{array}[c]{cc}%
        0 & \overline{F_{-}(x)}\\
        F_{-}(x) & 0
            \end{array}\right), \label{eq.Omega-}\\
    \Omega_{+}(x)  &  :=\left(\begin{array}[c]{cc}%
        0 & F_{+}(x)\\
        \overline{F_{+}(x)} & 0
            \end{array}\right). \label{eq.Omega+}%
\end{eqnarray*}
We will prove:

\begin{proposition}
Suppose that $u\in X$; then the kernels
$\Gamma_{+}$ and $\Gamma_{-}$ obey the equations%
\begin{eqnarray}\fl
     \Gamma_{-}(x,\zeta)+\Omega_{-}(x+\zeta)
        +\int_{-\infty}^{0}\Gamma_{-}(x,t)\Omega_{-}(x+t+\zeta)\,\rmd t &=0
        & \quad  {\it for\ a.e.} \quad \zeta <0,\label{eq.GLM-} \\
\fl  \Gamma_{+}(x,\zeta)+\Omega_{+}(x+\zeta)
        +\int_{0}^{\infty} \Gamma_{+}(x,t)\Omega_{+}(x+t+\zeta)\,\rmd t &=0
        &\quad  {\it for\ a.e.} \quad  \zeta>0. \label{eq.GLM+}%
\end{eqnarray}

\end{proposition}

The equations (\ref{eq.GLM-}) and (\ref{eq.GLM+}) are the left and right
Gelfand--Levitan--Marchenko equations.

\begin{proof}
We will derive (\ref{eq.GLM-}); the derivation of (\ref{eq.GLM+}) is similar
and will be omitted. We begin with the identity%
\[
    \Psi_{+}(x,z)=\Psi_{-}(x,z)\
        \left(\begin{array}[c]{cc}%
            a(z) & \overline{b}(z)\\
            b(z) & \overline{a}(z)
        \end{array}\right),
\]
for a fixed~$x$ and real~$z$, which implies that%
\begin{eqnarray}
\frac{m_{1}^{+}}{a}
    &=m_{1}^{-}+r_{-}\rme^{-\rmi zx}m_{2}^{-},
        \label{eq.sc-.h2+}\\
\frac{m_{2}^{+}}{\overline{a}}
    &=m_{2}^{-}+\overline{r_{-}}\,\rme^{\rmi zx}m_{1}^{-}.
        \label{eq.sc-.h2-}%
\end{eqnarray}
From Proposition \ref{prop.Hardy.m} and Remark \ref{rem.a}, the left-hand side
of (\ref{eq.sc-.h2+}) satisfies%
\[
\frac{m_{1}^{+}}{a}-\left(
\begin{array}
[c]{c}%
1\\
0
\end{array}
\right)  \in H_{+}^{2}(\mathbb{R},\mathbb{C}^{2})
\]
for each $x$. On the other hand,
it follows from Proposition~\ref{prop.Gamma.rep} and (\ref{eq.F-}) that
\[
        m_{1}^{-}+r_{-}\rme^{-\rmi zx}m_{2}^{-}
            -\left(\begin{array}[c]{c}1\\0 \end{array}\right)
\]
is represented as%
\begin{eqnarray}\label{eq.sc-.h2+.rhs.rep}
    \eqalign{ \int_{-\infty}^{\infty}\Biggl[ \Gamma_{1}^{-}(x,\zeta)
        &  + F_{-}(x+\zeta) \left(\begin{array}[c]{c}0\\1\end{array}\right)\Biggr] \rme^{\rmi z\zeta}\,\rmd\zeta \\
        &  + \int_{-\infty}^{\infty} \left[
            \int_{-\infty}^{0}\Gamma_{2}^{-}(x,t)F_{-}(x+t+\zeta)\,\rmd t\right]  \rme^{\rmi z\zeta}\,\rmd\zeta,}
\end{eqnarray}
where we regard the columns~$\Gamma_{1}^{-}(x,\,\cdot\,)$ and~$\Gamma_{2}^{-}(x,\,\cdot\,)$ of~$\Gamma_-(x,\,\cdot\,)$ as
vector-valued functions on the line which vanish for $\zeta>0$. Since~(\ref{eq.sc-.h2+.rhs.rep}) should give a function in~$H_+^2(\mathbb R,\mathbb C^2)$, we get
\begin{equation}\label{eq.GLM-.col1}%
    \Gamma_{1}^{-}(x,\zeta) + F_{-}(x+\zeta)
        \left(\begin{array}[c]{c}0\\1\end{array}\right)  + \int_{-\infty}^{0}\Gamma_{2}^{-}(x,t)F_{-}(x+t+\zeta)\rmd t=0
\end{equation}
for almost every $\zeta<0$. A similar calculation with~(\ref{eq.sc-.h2-})
shows that%
\begin{equation}\label{eq.GLM-.col2}%
    \Gamma_{2}^{-}(x,\zeta) + \overline{F_{-}(x+\zeta)}
            \left(\begin{array}[c]{c} 1\\ 0 \end{array}\right)  +\int_{-\infty}^{0}\Gamma_{1}^{-}(x,t)
                \overline{F_{-}(x+t+\zeta)}\,\rmd t=0
\end{equation}
for almost every $\zeta<0$. \ The equations (\ref{eq.GLM-.col1}) and
(\ref{eq.GLM-.col2}) together imply (\ref{eq.GLM-}).
\end{proof}

Finally we recall the formulas for reconstructing $u$ from $\Gamma_{\pm}$,
which play a crucial role in the inverse theory.

\begin{proposition}
Suppose that $u\in S(\mathbb{R})$. Then%
\begin{equation}
u(x)=-\Gamma_{12}^{+}(x,0) \label{eq.u.recon.+}%
\end{equation}
and
\begin{equation}
u(x)=\Gamma_{12}^{-}(x,0). \label{eq.u.recon.-}%
\end{equation}
\end{proposition}

\begin{proof}
If $u\in S(\mathbb{R})$, it is easy to see that the kernels
$\Gamma_{\pm}(x,\zeta)$ are smooth functions of $x$ and $\zeta$, and that the
derivatives are integrable in $\zeta$. Using the integral representations for
$\mathrm{M}_{\pm}$ in terms of $\Gamma_{\pm}$ together with the integration by parts
formula%
\[
    \int_{0}^{\infty} f(\zeta)\exp(\rmi z\zeta\sigma_{3}) \rmd\zeta
        = \sum_{j=1}^{k} \left( \frac{-1}{iz}\sigma_{3}\right)^{j}
            f^{(j-1)}(0)+ \Or \left( \frac{1}{z^{k+1}}\right)  ,
\]
true for $f\in C^{k+1}([0,\infty);\mathrm{M}_{2}(\mathbb{C}))$ with
$\int\left\vert f^{(j)}(\zeta)\right\vert \rmd\zeta<\infty$ for $0\leq j\leq
k+1$, it is easy to see from the integral representations (\ref{eq.m+.rep})--(\ref{eq.m-.rep}) that $\mathrm{M}_{+}$ and $\mathrm{M}_{-}$ have differentiable asymptotic expansions of the form%
\begin{equation}
    \mathrm{M}_{\pm}(x,z) \sim \mathrm{I} + \frac{\mathrm{M}_{\pm,1}(x)}{z} +
        \Or\left(  \frac{1}{z^{2}}\right)  \label{eq.M.asym}%
\end{equation}
for $u\in S(\mathbb{R})$, where%
\[
    \mathrm{M}_{\pm,1}(x)=\pm\rmi\Gamma_{\pm}(x,0)\sigma_{3}.
\]
Substituting~(\ref{eq.M.asym}) into~(\ref{eq.m}), we conclude that%
\begin{eqnarray*}
    \mathrm{Q}(x)&  =-\rmi\,\ad_\sigma (\mathrm{M}_{\pm,1}(x))\\
        &  =\pm\case12\left[ \sigma_{3}\Gamma_{\pm}(x,0)\sigma_{3}
                -\Gamma_{\pm}(x,0)\right].
\end{eqnarray*}
The formulas (\ref{eq.u.recon.+}) and (\ref{eq.u.recon.-}) are an immediate consequence.
\end{proof}

\section{Inverse Scattering}

\label{sec.inverse}

We now suppose given $\left(  F_{+},F_{-}\right)  \in X_{1}\times X_{1}$ with
$F_{-}=\mathcal{F_-}(\mathcal{I}\widehat{F}_{+})$ and show that the equations~(\ref{eq.GLM-}) and (\ref{eq.GLM+}) admit unique solutions~$\Gamma_\pm$ depending continuously on $F_{\pm}\in X_{1}$. We then argue by continuity and density that the reconstruction formulas (\ref{eq.u.recon.+}) and (\ref{eq.u.recon.-}), known from the classical theory for $F_\pm$ in the Schwartz class $S(\mathbb R)$, continue to hold also for $F_\pm \in X_1$. We remark that equations~(\ref{eq.GLM-}) and (\ref{eq.GLM+}) can also be analyzed using the factorization theory in operator algebras as developed e.g. in~\cite[Ch.~IV]{GK:1970}; see also~\cite{Mykytyuk:2003,Mykytyuk:2004} for some extensions and~\cite{AHM:2005} for an example of application in inverse problems for Dirac operators in the AKNS form.

We give the details for recovering $u$ from $F_{+}$ since the recovery of
$u$ from $F_{-}$ is closely analogous. By (\ref{eq.u.recon.+}), it will
suffice to solve~\eref{eq.GLM+} for $\Gamma_{12}^{+}$ and show that the solution is
sufficiently regular so that $u(x)=-\Gamma_{12}^{+}(x,0)$ is well-defined and
gives a continuous map from~$X_{1}$ to~$X_c^+$ for any $c\in\mathbb{R}$.

First we prove the existence of a unique solution. From (\ref{eq.GLM+}) we
have (writing $F$ for $F_{+}$)%
\begin{eqnarray*}
    F(x+\zeta)+ &\Gamma_{12}^{+}(x,\zeta)
        +\int_{0}^{\infty}\Gamma_{11}^{+}(x,t)F(x+t+\zeta)\,\rmd t   =0,\\
    &\Gamma_{11}^{+}(x,\zeta)
        +\int_{0}^{\infty}\Gamma_{12}^{+}(x,t)
                            \overline {F}(x+t+\zeta)\,\rmd t   =0.
\end{eqnarray*}
Iterating these equations and writing $\Gamma$ for $\Gamma_{12}^{+}$, we see
that
\begin{equation}\label{eq.Gamma}%
\fl
    \Gamma(x,\zeta) + F(x+\zeta)
        -\int_{0}^{\infty}\int_{0}^{\infty}\Gamma(x,t_{2})
            \overline{F}(x+t_{2}+t_{1})F(x+t_{1}+\zeta)\,
                \rmd t_{2}\,\rmd t_{1} =0.
\end{equation}
To solve this equation, we consider the scalar operators $T_{F}(x)$ and
$T_{\overline{F}}(x)$ on $L^{1}(\mathbb{R}^{+})$ given by%
\begin{eqnarray*}
    \left[ T_{F}(x)\psi\right](\zeta)
        & =\int_{0}^{\infty}\psi(t)F(x +t+\zeta)\,\rmd t,\\
    \left[ T_{\overline{F}}(x)\psi\right](\zeta)
        & =\int_{0}^{\infty}\psi(t)\overline{F}(x+t+\zeta)\,\rmd t
\end{eqnarray*}
and regard~(\ref{eq.Gamma}) as an equation in $L^{1}(\mathbb{R}^{+})$
for each fixed $x$:
\begin{equation}\label{eq.Gamma.L1}%
    \Gamma(x,\,\cdot\,) + F(x+\,\cdot\,)
        - \left(T_{F}(x)\circ T_{\overline{F}}(x)\right)
            \Gamma(x,\,\cdot\,)=0.
\end{equation}
The operators $T_{F}(x)$ and $T_{\overline{F}}(x)$ satisfy the estimates%
\begin{equation}\label{eq.TF.L2}%
    \eqalign{
    \| T_{F}(x)\|_{L^{1} \rightarrow L^{1}},\,
    \| T_{\overline{F}}(x)\|_{L^{1}\rightarrow L^{1}}
        &  \leq \|F\|_{1},\\ 
    \| T_{F}(x)\|_{L^{1}\rightarrow L^{2}},\,
    \|T_{\overline{F}}(x)\|_{L^{1}\rightarrow L^{2}}
        &  \leq\|F\|_{X},\\ 
    \|T_{F}(x)\|_{L^{2}\rightarrow L^{2}},\,
    \|T_{\overline{F}}(x)\|_{L^{2}\rightarrow L^{2}}
        &  \leq\|r\|_{\infty}<1,
        }
\end{equation}
where $r=\widehat{F}$. Moreover, the map
\[
    F\mapsto T_{F}%
\]
is a continuous map from $X$ into the bounded operators from $L^{1}(\mathbb{R}^+)$ to itself, from $L^{1}(\mathbb{R}^+)$ to $L^{2}(\mathbb{R}^+)$, and from $L^{2}(\mathbb{R}^+)$ to itself. Since $T_{F}$ is
compact as an operator from $L^{p}(\mathbb{R}^+)$ to itself for $F\in C_{0}^{\infty}(\mathbb{R})$ and $p=1,2$, it follows by density that the same is
true for any $F\in X$.

The $L^{2}$ estimates imply that the operator%
\[
    \bigl(I-T_{F}(x)\circ T_{\overline{F}}(x)\bigr)^{-1}%
\]
exists as a bounded operator from $L^{2}(\mathbb{R}^{+})$ to itself given by a
convergent Neumann series. In particular, $\ker_{L^{2}(\mathbb{R}^{+})}\left(
I-T_{F}(x)\circ T_{\overline{F}}(x)\right)  $ is trivial. On the other hand,
any solution $\psi\in L^{1}(\mathbb{R}^{+})$ of the equation
\[
   \bigl(  I-T_{F}(x)\circ T_{\overline{F}}(x)\bigr)  \psi=0
\]
actually belongs to $L^{2}(\mathbb{R}^{+})$ since $T_{F}(x)\circ T_{\overline{F}}(x)$ maps $L^{1}(\mathbb{R}^+)$ into $L^{2}(\mathbb{R}^+)$. It follows that $\ker
_{L^{1}(\mathbb{R}^{+})}\left(  I-T_{F}(x)\circ T_{\overline{F}}(x)\right)  $
is also trivial. Since $T_{F}$ and $T_{\overline{F}}$ are compact, we conclude
that $\left(  I-T_{F}(x)\circ T_{\overline{F}}(x)\right)  ^{-1}$ also exists
as a bounded operator from $L^{1}(\mathbb{R}^+)$ to $L^{1}(\mathbb{R}^+)$.

The solution to~(\ref{eq.Gamma.L1}) is given by%
\begin{eqnarray}\label{eq.Gammax}
    \Gamma(x,\,\cdot\,)
         = - \left(  I-T_{F}(x)\circ T_{\overline{F}}(x)\right)^{-1}F_x
         =: -F_x - G(x,\,\cdot\,),
\end{eqnarray}
where~$F_x(\,\cdot\,):=F(x + \,\cdot\,)$ and
\[
    G(x,\,\cdot\,) :=
        T_{F}(x)\left(  I-T_{\overline{F}}(x)\circ T_{F}(x)\right)^{-1}
        \left[T_{\overline{F}}(x)F_x\right].
\]
The right-hand side of (\ref{eq.Gammax}) defines for each $x\in\mathbb{R}$ a
function $\Gamma(x,\,\cdot\,)\in X_0^+$.

To study continuity in $x$ and $\zeta$, set
\begin{equation}
    H(x,\,\cdot\,):= \left(I-T_{\overline{F}}(x)\circ T_{F}(x)\right)^{-1}
        \left[T_{\overline{F}}(x)F_x\right]. \label{eq.H}%
\end{equation}
The estimate%
\[
    \| H(x,\,\cdot\,)\|_{L^{2}(\mathbb{R}^{+})}
        \leq \bigl( 1-\| r\| _{\infty}^{2}\bigr)^{-1}
            \|r\|_{\infty} \|F\|_{2}%
\]
holds, and $x\mapsto H(x,\,\cdot\,)$ is a continuous mapping from $\mathbb{R}$
into $L^{2}(\mathbb{R}^{+})$, as follows from the continuity of~$T_{F}(x)$ as an
operator-valued function of $x$, the continuity of $x\mapsto F_x$ as
a mapping from~$\mathbb R$ into~$L^2(\mathbb{R}^+)$, and the uniform bounds on the resolvent. From the formula%
\[
    G(x,\zeta)=\int_{0}^{\infty}H(x,t)F(x+t+\zeta)\,\rmd t,
\]
the Schwarz inequality, and the continuity of $H(x,\,\cdot\,)$ as a mapping into
$L^{2}(\mathbb{R}^{+})$, we deduce that $G(x,\zeta)$ is jointly continuous in
$x$ and $\zeta$ and is uniformly bounded. Thus:

\begin{lemma}
Suppose that $F\in X_{1}$. For each $x\in\mathbb{R}$, there exists a unique
so\-lu\-tion~$\Gamma(x,\,\cdot\,)$ of~(\ref{eq.Gamma}) belonging to~$X_0^+$. Moreover,%
\[
    \Gamma(x,\zeta)= -F(x+\zeta) - G(x,\zeta),
\]
where~$G$ is a bounded, jointly continuous function of $x$ and $\zeta$.
\end{lemma}

We can now obtain a limit of $\Gamma(x,\zeta)$ as $\zeta\to0$ and compute a putative
reconstruction
\begin{equation}
    u_{+}(x):=-\Gamma(x,0) = F(x)+w(x), \label{eq.u+}%
\end{equation}
where%
\begin{equation}
    w(x) = \int_{0}^{\infty}H(x,t)F(x+t)\,\rmd t. \label{eq.w}%
\end{equation}

Recall that $X_{c}^{+}=L^{1}(c,\infty)\cap L^{2}(c,\infty)$. \ We will now
show that, for any fixed~$c\in\mathbb{R}$,  the map
\[
    \mathcal{G}_{+}  \,:\, X_{1} \ni F \mapsto u_{+} \in X_{c}^{+}
\]
defined by~(\ref{eq.u+}) is bounded and continuous. We will also show that
\begin{equation}\label{eq.G+}%
    \mathcal{G}_{+}(F)= F + \sum_{n=1}^{\infty}
            \mathcal{G}_{2n+1}^{+}(F,\overline{F})
\end{equation}
for multilinear maps $\mathcal{G}_{n}^{+}:X_{1}^{n}\rightarrow X_{c}^{+}$.
Indeed, set
\[
    H_{2n}(x,\,\cdot\,)
        := \left(T_{\overline{F}}(x)\circ T_{F}(x)\right)^{n-1}
            \left[T_{\overline{F}}(x)F_x\right]
\]
and%
\[
    \mathcal{G}^+_{2n+1}(F,\overline{F})
        = \int_{0}^{\infty}F(x+t)H_{2n}(x,t)\,\rmd t.
\]
Then the claimed representation follows from (\ref{eq.H}), (\ref{eq.u+}), and
(\ref{eq.w}) provided that the series (\ref{eq.G+}) converges in $X_{c}^{+}$.

To prove convergence, suppose that $F\in X_{1}$ and choose~$\rho<1$ such
that~$\|\widehat F\|_\infty<\rho$. First, from (\ref{eq.TF.L2}) we have,
\[
    \| H_{2n}(x,\,\cdot\,)\|_{2} \leq \rho^{2n-1} \|F\|_{2},
\]
so that by the Schwarz inequality%
\begin{equation}\label{eq.G2n+1.infty}%
    \| \mathcal{G}^+_{2n+1}(F,\overline{F})\| _{\infty}
        \leq \rho^{2n-1}\| F\| _{2}^{2}.
\end{equation}
It follows that $\sum_{n=1}^{\infty}\mathcal{G}_{2n+1}^{+}(F,\overline{F})$
converges in~$L^{\infty}({\mathbb R})$. On the other hand, we have the explicit formula
\begin{eqnarray*}\fl
    \mathcal{G}^+_{2n+1}(F,\overline{F})(x)
        =\int_{\mathbb{R}_{+}^{2n}}F(x+t_1)\overline{F}(x+t_1+t_{2})
            \cdots \overline{F}(x+t_{2n-1}+t_{2n})F(x+t_{2n})\,
            \rmd \mathrm{t},
\end{eqnarray*}
with~$\mathrm{t}:=(t_1,\dots,t_{2n}) \in{\mathbb R}_+^{2n}$.
Fix~$F$ and choose~$x_{0}$ so that%
\begin{equation*}
    \int_{x_{0}}^{\infty} |F(s)|\,\rmd s < \rho. 
\end{equation*}
Note that the same condition holds for some relatively open set of~$F\in X_{1}$.
For any~$F$ in such an open set we then have%
\begin{equation}\label{eq.G2n+1.L1}%
    \| \mathcal{G}^+_{2n+1}\| _{L^{1}(x_{0},\infty)}
        \leq\rho^{2n+1}.
\end{equation}
Thus, combining (\ref{eq.G2n+1.infty}) and (\ref{eq.G2n+1.L1}), we have for
any $c\in\mathbb{R}$ the estimate%
\begin{equation}\label{eq.G2n+1.L1a}
    \| \mathcal{G}^+_{2n+1}\| _{L^{1}(c,\infty)}
        \leq(\max\{x_{0}-c,0\}\|F\|_{2}^{2} + \rho^{2})\rho^{2n-1},
\end{equation}
which shows that $\sum_{n=1}^{\infty}\mathcal{G}_{2n+1}^{+}(F,\overline{F})$
converges in $L^{1}(c,\infty)$ for any fixed $c\in\mathbb{R}$.
It now follows from~(\ref{eq.G2n+1.infty}) and~(\ref{eq.G2n+1.L1a}) that the series~(\ref{eq.G+}) converges in $L^{2}(c,\infty)$, and hence in~$X_{c}^{+}$ as claimed.

We have shown:

\begin{lemma}\label{lemma.GLM+}
For any fixed $c\in\mathbb{R}$, the
mapping~$\mathcal{G}_{+}\,:\,X_{1}\rightarrow X_{c}^{+}$ is continuous and admits the representation~(\ref{eq.G+}). Moreover, $\mathcal{G}_{+}-I$ is a continuous map of~$X_{1}$ into~$C(\mathbb{R})$.
\end{lemma}

Similarly, we can study the ``left''  Gelfand--Levitan--Marchenko equation~(\ref{eq.GLM-}) to define a map
$\mathcal{G}_{-}\,:\,F_{-}\mapsto u_{-}$ via the reconstruction formula~(\ref{eq.u.recon.-}). The analogue of~(\ref{eq.Gamma}) (where now~$\Gamma$
denotes~$\Gamma_{12}^{-}$ and~$F$ denotes~$F_{-}$) is
\[ \fl
    \Gamma(x,\zeta)+F(x+\zeta)
        -\int_{-\infty}^{0}\int_{-\infty}^{0}
        \Gamma(x,t_{2})\overline{F}(x+t_{2}+t_{1})F(x+t_{1}+\zeta)
            \,\rmd t_{2}\,\rmd t_{1} = 0,
\]
and we set
\[
    u_{-}(x)=\lim_{y\uparrow0}\Gamma(x,y)
\]
where the limit is taken in $X_c^-$. In an analogous way we obtain a
representation%
\begin{equation}\label{eq.G-}%
    u_- = \mathcal{G}_{-}(F)
        :=-F+\sum_{n=1}^{\infty} \mathcal{G}_{2n+1}^{-}
            (F,\overline{F}),
\end{equation}
where $\mathcal{G}_{n}^{-}$ are multilinear functions acting from~$(X_{1})^{n}$ to~$X_{c}^{-}$. An analogous argument shows:

\begin{lemma}
\label{lemma.GLM-}For any fixed $c\in\mathbb{R}$, the mapping
$\mathcal{G}_{-}:X_{1}\rightarrow X_{c}^{-}$ is continuous and admits the representation~(\ref{eq.G-}). Moreover, $\mathcal{G}_{-}-I$ is a continuous map of~$X_{1}$ into~$C(\mathbb{R})$.
\end{lemma}

With these continuity statements, we can prove Theorem \ref{thm.1}.

\begin{proofof}{Proof of Theorem~\ref{thm.1}}
We will discuss only the mapping $\mathcal{F}_+\circ\mathcal{S}_{+}$. To show it is injective, assume that $F_+\in X_1$ corresponds to~$u\in X$. Then, solving the Marchenko equation~\eref{eq.GLM+}, we find the kernel $\Gamma_+$ that by~\eref{eq.m+.rep} determines uniquely the Jost solution $\Psi_+$, which, in turn, gives the matrix potential~$\mathrm{Q}$ of~\eref{eq.Q}. Therefore different $u$ are mapped into different $F_+$, so that there is at most one potential~$u\in X$ with a given right reflection coefficient.

We next construct the inverse of~$\mathcal{F}_+\circ\mathcal{S}_{+}$.
Given $F$ in~$X_1$, we set $u_{+}=\mathcal{G}_{+}F$ and $u_{-} = \mathcal{G}_{-}\bigl(\mathcal{F}_- (\mathcal{I} \widehat F)\bigr)$. We need to show the following facts:
\begin{itemize}
\item[(1)] $u_{+}=u_{-}$, and
\item[(2)] the potential $u=u_{+}=u_{-}$ has reflection coefficients $r_{+}%
=\widehat{F}$ and $r_{-}=\mathcal{I}\widehat{F}$.
\end{itemize}
Then continuity of the mapping $F \mapsto u$ follows from those of $\mathcal G_\pm$ and~(1).  Analyticity of this map follows similarly from~(1), analyticity of $\mathcal G_-$ on $X_1^-$ and $\mathcal G_+$ on $X_{-1}^+$, and the
representation~$u = \zeta u_- + (1-\zeta) u_+$ for a function~$\zeta \in C^\infty$ with~$\zeta(t) =1$ for~$t<-1$ and~$\zeta(t)=0$ for~$t>1$. Fact~(2)
shows that the range of~$\mathcal{F}_+\circ\mathcal{S}_{+}$ is~$X_{1}$.

We will appeal to standard results for the Gelfand--Levitan--Marchenko equations
when $F\in S(\mathbb{R})$ (see for example Section~II.\,4 of
\cite{FT:2007}) and the continuity of the direct and inverse scattering maps established above. Suppose given $F\in X_{1}$ and
let $\left(F_{n}\right)$ be a sequence from $S(\mathbb{R})$ with
$F_{n}\rightarrow F$ in $X$. Note that, for $n$ sufficiently large, $F_{n}\in
X_{1}$ since $\widehat{F}_{n}\rightarrow\widehat{F}$ in $L^{\infty}%
(\mathbb{R})$. Let $F_{+}:=F$, define $F_{-}$ as $\mathcal{F}_-
(\mathcal{I}\widehat{F})$, and similarly let $F_{n,+}:=F_{n}$ and define
$F_{n,-}$ as $\mathcal{F}_-(\mathcal{I}\widehat{F}_{n})$. \

Now set $u_{n,\pm}:=\mathcal{G}_{\pm}F_{n,\pm}$. By standard theory, we have $u_{\pm,n}\in S(\mathbb{R})$ with $u_{+,n}=u_{-,n}=u_{n}$, and the potential
$u_{n}$ has associated reflection coefficients $\widehat{F}_{n,+}(s)$ and~$\widehat{F}_{n,-}(-s)$. By Lemmas~\ref{lemma.GLM+} and \ref{lemma.GLM-},
$u_{n,+}\rightarrow u_{+}$ in~$X_{c}^+$  and $u_{n,-}\rightarrow u_{-}$ in $X_{c}^-$, so that $u_{+}=u_{-}=u$ and $u_{n}\rightarrow u$ in $X$. By the continuity of the direct scattering maps of Lemma \ref{lemma.direct}, we can also conclude that $u$ has reflection coefficients $r_+(s)=\widehat{F}_{+}(s)$ and $r_-(s)=\widehat{F}_{-}(-s)=(\mathcal{I}\widehat{F})(s)$.
\end{proofof}

Now it is easy to deduce Corollary \ref{cor.Schroed}.

\begin{proofof}{Proof of Corollary \ref{cor.Schroed}}
We only need to prove that the maps are
real-analytic. Real analyticity of the direct maps follows immediately from Lemma~\ref{lemma.direct} and Remark~\ref{rem.real}. Real analyticity of the inverse
map follows similarly.
\end{proofof}

\ack
This material is based upon work supported by the National Science Foundation under Grant DMS-0710477 (CF, RH, and PP) and by the Deutsche Forschungsgemeinschaft under project 436 UKR 113/84 (RH and YM). RH acknowledges support from the College of Arts and Sciences at the University of Kentucky and thanks the Department of Mathematics at the University of Kentucky for hospitality during his stay there. RH, YM, and PP thank the Institut f\"ur angewandte Mathematik der Universit\"at Bonn for hospitality during part of the time that this work was done. PP thanks Percy Deift for helpful conversations and SFB 611 for financial support of his research visit to Universit\"at Bonn.

\appendix

\section{}\label{app.Wiener}

\def\thesection{\Alph{section}}

\setcounter{section}{1}

Let us denote by $\widehat{L^1}(\mathbb{R})$ the Wiener algebra of Fourier transforms~\eref{eq.Fourier} of functions in $L^{1}(\mathbb{R})$ with
norm~$\|\hat{f}\|_{\widehat{L^1}}:=\| f\|_{L^1}$, and by $\widehat{X}$ the Banach algebra that is the image of~$X = L^{1}(\mathbb{R})\cap L^{2}(\mathbb{R})$ under the Fourier transform,
equipped with the norm $\| \hat{f}\| _{\widehat{X}}:=\|f\|_{X}$. We also denote by~$1 \dotplus\widehat{X}$ the unital extension of~$\widehat{X}$ obtained by adding the constant functions and norming $1 \dotplus\widehat{X}$ with the norm%
\[
    \| c+\hat{f}\|_{1 \dotplus\widehat{X}}
        = |c| + \| \hat{f}\| _{\widehat{X}};
\]
we similarly define~$1\dotplus \widehat{L^1}(\mathbb{R})$.
The Fourier transform extends to $1 \dotplus\widehat{X}$ by mapping the
constant $1$ into the convolution identity~$\delta$.

We will need the following results.

\begin{lemma}
Suppose that $f=\alpha+\widehat{h}\in 1 \dotplus\widehat{X}$. Then $f$ is invertible
in the Banach algebra~$1 \dotplus\widehat{X}$ if and only if $f$ is non-vanishing on $\mathbb{R}$ and
$\alpha\neq0$.
\end{lemma}

\begin{proof}
If $f$ is invertible in $1 \dotplus\widehat{X}$ it is also invertible in
$1 \dotplus\widehat{L^1}(\mathbb{R})$, so the condition is necessary by the Wiener theorem. If, on the other hand, $f$ does not vanish on $\mathbb{R}$ and $\alpha\neq0$, then~$f$ is invertible in~$1 \dotplus\widehat{L^1}(\mathbb{R})$ with $f^{-1}=\alpha^{-1}+\widehat{g}$ for~$g\in L^{1}(\mathbb{R})$. We need to check that $g\in L^{2}(\mathbb{R})$. Without loss we take $\alpha=1$ and compute that $\widehat{g}=-(  1+\widehat{h})^{-1}\widehat{h}$, which shows that~$g\in L^{2}(\mathbb{R})$ as required.
\end{proof}

We now have an analogue of the Wiener--Levi theorem for $1 \dotplus\widehat{X}$.

\begin{lemma}\label{lem:WL}
Assume that $f\in 1 \dotplus\widehat{X}$ and that $\phi$ is a function that
is analytic in an open neighborhood $\Omega$ of the closure of the range
of~$f$. Then $\phi\circ f\in 1 \dotplus\widehat{X}$ and, moreover, the map
\[
    f\mapsto\phi\circ f
\]
is an analytic map from $1 \dotplus\widehat{X}$ into itself when restricted to
functions with range contained in a fixed compact subset of $\Omega$.
\end{lemma}

\begin{proof}
It suffices to note that, according to the above, the spectrum of~$f$
in~$1 \dotplus\widehat{X}$ coincides with the closure of its range. Then the
standard functional calculus for Banach algebras applies, thus yielding the result.
\end{proof}

\begin{lemma}\label{lem.anal}
Assume that $U$ is an open subset of $1\dotplus \widehat{X}$ and that  $\mathit{\Psi} :U \to 1 \dotplus \widehat{X}$ is analytic in $f$ and $\overline{f}$ in the sense of Definition 2.1. Let also $\phi$ be a complex-valued function that is analytic in a neighbourhood of the closure of the set
\(
	\bigcup_{f\in U} \Ran \mathit{\Psi}(f).
\)
Then the map $f \mapsto \phi\circ\mathit{\Psi}(f)$ from $U$ into $1\dotplus \widehat{X}$ is analytic in~$f$ and $\overline{f}$.
\end{lemma}

\begin{proof} By Lemma~\ref{lem:WL}, $g \mapsto \phi\circ g$ is an analytic mapping in the Banach algebra $1 \dotplus \widehat{X}$ defined on $\mathit{\Psi}(U)$. It is therefore analytic in $1 \dotplus\widehat{X}$ considered as a Banach space. Thus the mapping $f \mapsto \phi \circ \mathit{\Psi}(f)$ is analytic in $f$ and $\overline{f}$ as a composition of analytic mappings between Banach spaces.
\end{proof}

\end{document}